\documentclass[12pt,twoside,a4paper]{article}

\usepackage[active]{srcltx} 
\usepackage{hyperref}
\usepackage[T1]{fontenc}
\usepackage{amsmath,amsfonts,amsthm,amssymb}
\DeclareFontEncoding{LS1}{}{}
\DeclareFontSubstitution{LS1}{stix}{m}{n}
\DeclareSymbolFont{symbols2}{LS1}{stixfrak} {m} {n}
\DeclareMathSymbol{\operp}{\mathbin}{symbols2}{"A8}

\title{Hodge Operators and Exceptional Isomorphisms\\
  between Unitary Groups}
\author{Linus Kramer, Markus J. Stroppel}%
\def\shortauthor{Linus Kramer, Markus J. Stroppel}%
\def\shorttitle{Hodge Operators and Exceptional Isomorphisms}                                     

\newcounter{rememberEnumi}

\usepackage[all]{xy}

\advance\textheight by \topskip

\makeatletter
\newcommand{\rightheadline}{\hfill\sc\shorttitle\hfill}
\newcommand{\leftheadline}{\hfill\sc\shortauthor\hfill}

\def\ps@headings{\let\@mkboth\@gobbletwo
\def\@oddhead{\rightheadline\rm\thepage}
\def\@oddfoot{}
\def\@evenhead{\rm\thepage\leftheadline}
\def\@evenfoot{}
}
\renewcommand{\section}{%
	\@startsection{section}{1}{\z@}%
	   {-0.8truecm}%
	   {\medskipamount}%
	   {\boldmath\normalfont\bfseries \centering}}
\makeatother
         
\newtheoremstyle{mytheorem}
 {\baselineskip}
 {\baselineskip}
 {\itshape}
 {}
 {\bfseries}
 {.}
 {1.5em}
 {}

\newtheoremstyle{mydefinition}
 {\baselineskip}
 {\baselineskip}
 {}
 {}
 {\bfseries}
 {.}
 {1.5em}
 {}

\swapnumbers
\theoremstyle{mytheorem}
\newtheorem{theo}{Theorem}[section]
\newtheorem{lemm}[theo]{Lemma}

\newtheorem{prop}[theo]{Proposition}
\newtheorem{namet}[theo]{\myThmName}
\newenvironment{nthm}[1][\kern-.35em]{\edef\myThmName{#1}\begin{namet}}{\end{namet}}

\theoremstyle{mydefinition}
\newtheorem{defi}[theo]{Definition}

\newtheorem{exam}[theo]{Example}
\newtheorem{exas}[theo]{Examples}
\newtheorem{rema}[theo]{Remark}
\newtheorem{rems}[theo]{Remarks}
\newtheorem*{acks}{Acknowledgements}
\newtheorem{named}[theo]{\myThmName}
\newenvironment{ndef}[1][\kern-.35em]{\edef\myThmName{#1}\begin{named}}{\end{named}}
\makeatletter
\newcommand{\widebar}[1]{ {\mathchoice
    {\vbox
    {\m@th \ialign {##\crcr \noalign {\kern 1\p@ }\kern 1\p@ \hrulefill \crcr
        \noalign {\kern 1\p@ \nointerlineskip }%
        $\hfil \displaystyle {#1}\hfil $\crcr }}}
    {\vbox
    {\m@th \ialign {##\crcr \noalign {\kern 1\p@ }\kern 1\p@ \hrulefill \crcr
        \noalign {\kern 1\p@ \nointerlineskip }%
        $\hfil \textstyle {#1} $\crcr }}}
    {\vbox
    {\m@th \ialign {##\crcr \noalign {\kern 1\p@ }\kern 1\p@ \hrulefill \crcr
        \noalign {\kern 1\p@ \nointerlineskip }%
        $\hfil \scriptstyle {#1}\hfil $\crcr }}}
    {\vbox
    {\m@th \ialign {##\crcr \noalign {\kern 1\p@ }\kern 1\p@ \hrulefill \crcr
        \noalign {\kern 1\p@ \nointerlineskip }%
        $\hfil \scriptscriptstyle {#1}\hfil $\crcr }}}%
    }}
\makeatother
\let\gal\widebar
\newcommand{\bra}[1]{\langle#1\rangle}

\newcommand{\bj}{\boldsymbol{j}}

\newcommand{\Tns}[1][\,]{{\textstyle\bigotimes\nolimits^{#1}}}
\newcommand{\Ext}[1][\,]{{\textstyle\bigwedge\nolimits^{\!#1}}}
\newcommand{\coloneqq}{\mathrel{\mathop:}=}
\newcommand{\id}{\mathrm{id}}
\newcommand{\Char}{\operatorname{char}}
\let\phi\varphi

\newcommand{\disc}{\operatorname{disc}}
\newcommand{\transp}{{^\intercal}}

\makeatletter
\newcommand{\myBox}[1]{{\mathchoice
    {\vbox
    {\m@th \ialign {##\crcr \hline \crcr
        \noalign {\nointerlineskip }%
        \vrule
        \vphantom{x}\smash{\lower.21\p@\hbox{$\kern-1.7\p@ \displaystyle {#1}\kern-1.7\p@$}}\vrule\crcr
        \noalign {\nointerlineskip  }\crcr\hline%
      }}}
    {\vbox
    {\m@th \ialign {##\crcr \hline \crcr
        \noalign {\nointerlineskip }%
        \vrule
        \vphantom{x}\smash{\lower.25\p@\hbox{$\kern-1.7\p@ \textstyle {#1}\kern-1.7\p@$}}\vrule\crcr
        \noalign {\nointerlineskip  }\crcr\hline%
      }}}
    {\vbox
    {\m@th \ialign {##\crcr \hline \crcr
        \noalign {\nointerlineskip }%
        \vrule
        $\scriptstyle\vphantom{x}$\smash{\lower.25\p@\hbox{$\kern-1.7\p@ \scriptstyle {#1}\kern-1.6\p@$}}\vrule\crcr
        \noalign {\nointerlineskip  }\crcr\hline%
      }}}
    {\vbox
    {\m@th \ialign {##\crcr \hline \crcr
        \noalign {\nointerlineskip }%
        \vrule
        $\scriptscriptstyle\vphantom{x}$\smash{\lower.18\p@\hbox{$\kern-1.6\p@ \scriptscriptstyle {#1}\kern-1.5\p@$}}\vrule\crcr
        \noalign {\nointerlineskip  }\crcr\hline%
      }}}%
    }}
\makeatother
\newcommand{\myTimes}{{\vphantom{x}\smash{\times}}}
\newcommand{\sq}{\myBox{\phantom{\myTimes}}}
\newcommand{\sqt}{\myBox{\myTimes}}
\newcommand{\set}[2]{\left\{{#1}\left|\vphantom{#1#2\strut}\right.\, 
                    {#2}\right\}}
\newcommand{\smallset}[2]{\{{#1}\left|\vphantom{}\right.\, 
                    {#2}\}}
\newcommand{\Pu}[1]{\operatorname{Pu}(#1)}

\newcommand{\Aut}[1]{\mathrm{Aut}(#1)}

\newcommand{\Hom}[3][]{\mathrm{Hom}_{#1}(#2,#3)}
\newcommand{\End}[2][]{\mathrm{End}_{#1}(#2)}

\newcommand{\Type}[2]{\mathsf{#1}_{#2}}
\newcommand{\rType}[3]{\mathsf{#1}_{#2}^{#3}}
\newcommand{\CC}{\mathbb C}
\newcommand{\FF}{\mathbb F}
\newcommand{\HH}{\mathbb H}
\newcommand{\NN}{\mathbb N}
\newcommand{\OO}{\mathbb O}
\newcommand{\QQ}{\mathbb Q}
\newcommand{\RR}{\mathbb R}
\newcommand{\Ss}{\mathbb S}
\newcommand{\ZZ}{\mathbb Z}
\newcommand{\C}{{\boldsymbol{C}}}
\newcommand{\F}{{\boldsymbol{F}}}
\renewcommand{\H}{{\boldsymbol{H}}}
\newcommand{\K}{{\boldsymbol{K}}}
\renewcommand{\L}{{\boldsymbol{L}}}
\newcommand{\R}{{\boldsymbol{R}}}
\newcommand{\cP}{{\mathcal{P}}}
\newcommand{\cR}{{\mathcal{R}}}
\newcommand{\gL}[2][]{\Gamma\mathrm{L}_{#1}{(#2)}}
\newcommand{\GL}[2][]{\mathrm{GL}_{#1}{(#2)}}
\newcommand{\SL}[2][]{\mathrm{SL}_{#1}{(#2)}}

\newcommand{\PSL}[2][]{\mathrm{PSL}_{#1}{#2}}
\newcommand{\PGL}[2][]{\mathrm{PGL}_{#1}{#2}}

\newcommand{\Orth}[3][]{\mathrm{O\kern0pt}^{#1}(\ifx#2\empty\else#2,\fi#3)}
\newcommand{\Oprime}[3][]{\Omega^{#1}(\ifx#2\empty\else#2,\fi#3)}

\newcommand{\SO}[2]{\mathrm{SO}_{#1}{(#2)}}

\newcommand{\SOpV}[1]{\mathrm{SO}^{+}{(#1)}}
\newcommand{\OpmV}[1]{\Omega^{-}{(#1)}}
\newcommand{\OmV}[1]{\mathrm{O}^{-}{(#1)}}
\newcommand{\OV}[1]{\mathrm{O}{(#1)}}
\newcommand{\SOV}[1]{\mathrm{SO}{(#1)}}
\newcommand{\EOV}[1]{\mathrm{EO}{(#1)}}
\newcommand{\GOV}[1]{\mathrm{GO}{(#1)}}
\newcommand{\gOV}[1]{\mathrm{\Gamma O}{(#1)}}

\newcommand{\UV}[1]{\mathrm{U}{(#1)}}
\newcommand{\SUV}[1]{\mathrm{SU}{(#1)}}
\newcommand{\EUV}[1]{\mathrm{EU}{(#1)}}

\newcommand{\GUV}[1]{\mathrm{GU}{(#1)}}
\newcommand{\gUV}[1]{\Gamma\mathrm{U}{(#1)}}
\newcommand{\U}[3][\kappa]{\mathrm{U}^{#1}_{#2}{(#3)}}
\newcommand{\SU}[3][\kappa]{\mathrm{SU}^{#1}_{#2}{(#3)}}

\newcommand{\SaU}[2]{\mathrm{S\alpha U}_{#1}{(#2)}}

\newcommand{\SpV}[1]{\mathrm{Sp}(#1)}
\newcommand{\pf}{\mathrm{Pf}}
\newcommand{\pq}{\mathrm{Pq}}
\newcommand{\PG}[2][]{\mathrm{P}_{#1}(#2)}
\newcommand{\Gr}[3][]{\mathrm{Gr}_{#2,#1}(#3)}

\newcommand{\Abs}[1]{\mathrm{Abs}_{#1}}

\newcommand{\Cg}[1]{\operatorname{C}_{#1}}

\begin{document}
\maketitle

\begin{abstract}
  \noindent
  We give a generalization of the Hodge operator to spaces $(V,h)$
  endowed with a hermitian or symmetric bilinear form~$h$ over
  arbitrary fields, including the characteristic two case. Suitable
  exterior powers of~$V$ become free modules over an algebra~$K$
  defined using such an operator. This leads to several exceptional
  homomorphisms from unitary groups (with respect to~$h$) into groups
  of semi-similitudes with respect to a suitable form over some
  subfield of~$K$. The algebra $K$ depends on~$h$; it is a composition
  algebra unless $h$ is symmetric and the characteristic is two. %
\end{abstract}
Mathematics Subject Classification: 
{%
    \href{https://mathscinet.ams.org/mathscinet/search/mscbrowse.html?code=20G15}{20G15} 
    \href{https://mathscinet.ams.org/mathscinet/search/mscbrowse.html?code=20E32}{20E32} 
    \href{https://mathscinet.ams.org/mathscinet/search/mscbrowse.html?code=20G20}{20G20} 
    \href{https://mathscinet.ams.org/mathscinet/search/mscbrowse.html?code=20G40}{20G40} 
    \href{https://mathscinet.ams.org/mathscinet/search/mscbrowse.html?code=22C05}{22C05} 
    \href{https://mathscinet.ams.org/mathscinet/search/mscbrowse.html?code=11E39}{11E39} 
    \href{https://mathscinet.ams.org/mathscinet/search/mscbrowse.html?code=11E57}{11E57}
  }    

\section*{Introduction}

In many branches of mathematics, attempts at a systematic study (or
even a classification) of certain objects lead to the consideration of
families of examples constructed from given data (such as a
ground field, dimension, an invariant form, \dots).
With suitably defined categories, these constructions may even be
considered as functors. %
The reader may think of classical groups (data comprise a field, a
dimension, and an invariant form), or Dynkin diagrams (data comprise a
letter indicating a construction of a graph, and a
natural number giving the number of vertices in that graph),
or split forms of simple Lie algebras (data comprise a field, and a
Dynkin diagram),
or finite simple groups of Lie type (data again comprise a field, and a
Dynkin diagram),
or classical polar spaces (data comprise a field, a dimension, and a
polarity of the projective space of that dimension over that field).

While these functors usually assign non-isomorphic structures to
different data sets, it often happens that unexpected isomorphisms
occur for ``small'' data sets.
If such an exceptional isomorphism occurs, it usually induces an
exceptional situation in a related field, or can be explained and
understood by such a situation.

In the present paper, we are interested in exceptional isomorphisms
between certain classical groups. In this context, the simple groups
in question are usually obtained by forming first the commutator group
and then its quotient by the center, while the natural habitats of
classical groups are inside linear groups. For that reason, it is
sometimes better to study not only isomorphisms but also
homomorphisms.

We show that certain homomorphisms between classical groups can be
explained by variants of the Hodge operator.  Some of the isomorphisms
discussed here are not easily accessible by other methods, so we think
our considerations are worth the effort. %
We have tried to keep the presentation reasonably self-contained, in
particular, we avoid the use of elaborate results from Lie theory or
the theory of algebraic groups. Applications to real Lie groups are
made explicit in several examples (see~\ref{rem:SaU}, \ref{RR4},
\ref{ex:O2RR}, \ref{ex:U2CC}, \ref{SU4ContoSO6R},
\ref{SU4C2ontoSO6R2}, and~\ref{exam:Lorentz}); %
applications to finite groups are explicit
in~\ref{SU4finiteontoOminus6}, and \ref{exam:finiteOrth}.

\section{The Hodge operator}

\begin{ndef}[Notation]
  We consider a field $\F$ of arbitrary characteristic and a field
  automorphism $\sigma\colon x\mapsto\gal{x}$ whose square is the
  identity.  Let $\R$ denote the fixed field of~$\sigma$. So either
  $\F=\R$ or $\F|\R$ is a Galois extension of degree~$2$ (by Artin's
  Theorem, see~\cite[VI, Th.\,1.8]{MR1878556}).  %
  In any case we call $N_{\F|\R}\colon\F\to\R\colon x\mapsto \gal{x}x$
  the \emph{norm form} of $\F|\R$. Note that $N_{\F|\R}(\F)$ is just
  the set $\F^\sq \coloneqq \set{x^2}{x\in\F}$ of squares if
  $\sigma=\id$.

  A map $\lambda\colon V\to W$ between vector space over~$\F$ is
  called \emph{$\F$-semilinear} if~$\lambda$ is additive and there
  exists an automorphism~$\varphi$ of~$\F$ (called the companion
  of~$\lambda$) such that $\lambda(sv) = \varphi(s)\lambda(v)$ holds
  for each $s\in\F$ and each $v\in V$. If we want to state explicitly
  that the companion is~$\varphi$, we speak of an
  $\F$-$\varphi$-semilinear map. %
  (In the context of real and complex analysis, one encounters
  $\CC$-$\gal{\phantom{x}}$-semilinear maps under the name
  $\CC$-antilinear maps.)

  Let $(V,h)$ be an $n$-dimensional non-degenerate $\sigma$-hermitian
  space over~$\F$; i.e., a vector space~$V$ over~$F$ with a
  bi-additive map $h\colon V\times V\to F$ such that
  $h(v,ws) = h(v,w)s$ and $h(w,v)=\gal{h(v,w)}$ hold for all
  $v,w\in V$ and each $s\in F$. %
  (In particular, the map~$h$ is $\F$-linear in the right variable,
  and $\F$-$\sigma$-semilinear in the left variable.)

  If $\sigma=\id$ then the form~$h$ is a symmetric bilinear form. Such
  forms occur as polarizations of quadratic forms; i.e., as
  $f(x,y) \coloneqq q(x+y)-q(x)-q(y)$ for a quadratic form~$q$.  (Some
  authors refer to~$f$ as the polar form associated with~$q$.) %
  For any quadratic form in characteristic two, the polarization is an
  alternating form.

  If $\sigma=\id$ and $\Char\F=2$ we also assume that~$h$ is
  diagonalizable, i.e. that there exists an orthogonal basis. %
  We note that such a form will not occur as polarization of a
  quadratic form (unless it is zero). %
  In particular, the orthogonal group with respect to a
  \emph{bilinear} form on $\FF_{2^e}^n$ differs from the group usually
  denoted by $\Orth[\varepsilon]n{2^e}$. %
  (If $\Char\F\ne2$ or $\sigma\ne\id$ then diagonalizability is no
  extra condition, cf.~\cite[\S8,\,p.\,15]{MR0072144}
  or~\cite[I.3.4]{MR0506372}.)  %
  A vector $v\in V\smallsetminus\{0\}$ is called \emph{isotropic}
  (with respect to~$h$) if $h(v,v)=0$; the form~$h$ is called
  isotropic if there exists an isotropic vector.

  As we are going to extend the field~$\F$ to a non-commutative
  composition algebra (containing~$\F$ as a non-central subalgebra)
  over~$\R$, we have to be precise: elements of~$\F^n$ will be
  considered as columns, with scalars acting from the right, and
  matrices acting from the left. However, elements of~$\R$ (in
  particular, signs such as $(-1)^m$) will also occur on the left.

  As $\dim{V}$ is assumed to be finite, our assumption that~$h$ be not
  degenerate is equivalent to the fact that we obtain an
  $\F$-$\sigma$-semilinear isomorphism onto the dual space~$V^\vee$,
  namely
  \[
    h^\vee \colon V\to V^\vee\colon v\mapsto h(v,-) \,
  \]
  see e.g.~\cite[Ch.\,I, \S2]{MR1096299}.  Consider now the exterior
  algebra $\Ext{V}$, cf.~\cite[VI\,9]{MR2257570}. We note
  that~$\Ext{}$ is a functor on vector spaces and semilinear maps,
  cf.~\ref{functoriality} below.  Moreover, there is a natural
  isomorphism \(\textstyle%
  (\Ext{V})^\vee\cong\Ext(V^\vee) \,, \) %
  so we may write unambiguously~$\Ext{V}^\vee$. Explicitly, %
  we have \( %
  \bra{f_1\wedge\cdots\wedge f_\ell ,w_1\wedge\cdots\wedge
    w_\ell}=\det(\bra{f_i,w_j}) %
  \) %
  for $f_i\in V^\vee$ and $w_j\in V$, see~\cite[I.5.6]{MR0506372}
  or\footnote{ \ %
    The treatment in~\cite{MR0026989} is quite different from that in
    later editions~\cite{MR0274237}, \cite{MR979982}.} %
  \cite[\S\,8, Thme.\,1, p.\,102]{MR0026989}; %
  here $\bra{\cdot,\cdot}$ denotes the natural pairing between a
  vector space and its dual. %
  Applying the functor~$\Ext$ to $ h^\vee \colon V\to V^\vee$, we
  obtain $\Ext{h^\vee} \colon\Ext{V}\to\Ext{V^\vee}$; we interpret
  this as a hermitian form $\Ext{h}$ on the exterior
  algebra~$\Ext{V}$.  %
  Using the explicit formula above, we find
  \[\textstyle
    \Ext{h}(v_1\wedge\cdots\wedge v_\ell ,w_1\wedge\cdots\wedge w_\ell )
    =
    \Ext[\ell]{h}(v_1\wedge\cdots\wedge v_\ell ,w_1\wedge\cdots\wedge w_\ell )
    = \det(h(v_i,w_j)).
  \]
\end{ndef}

\begin{ndef}[The Pfaffian form]\label{def:Pfaffian}%
  The exterior algebra comes with a natural $\ZZ$-grading and
  $\Ext[n]V$ is $1$-dimensional. We fix an $\F$-linear isomorphism
  $b\colon \Ext[n]V\to\F$. For each positive integer~$\ell \le n$, the
  map~$b$ then induces an $\F$-linear isomorphism
  $\pf\colon \Ext[n-\ell]V\to\Ext[\ell]V^\vee$ given by
  \[
    \textstyle %
    \pf(v_1\wedge\cdots\wedge v_{n-\ell})(w_1\wedge\cdots\wedge w_\ell) %
    = b(v_1\wedge\cdots\wedge v_{n-\ell}\wedge w_1\wedge\cdots\wedge w_\ell) \,.
  \]
  This is the \emph{Pfaffian form}, see~\cite[VI\,10 Problems\,23--28,
  VIII\,12 Problem\,42]{MR2257570}.  The resulting bilinear map $\pf$
  on $\Ext{V}$ is ``graded symmetric'',
  \[
    \textstyle%
    \pf(v_1\wedge\cdots\wedge v_{n-\ell},w_1\wedge\cdots\wedge w_\ell ) %
    = (-1)^{(n-\ell )\ell} \, \pf(w_1\wedge\cdots\wedge
    w_\ell,v_1\wedge\cdots\wedge v_{n-\ell}) \,.
  \]
  If $n=2\ell$ then the bilinear form $\pf$ is isotropic on
  $\Ext[\ell]V$; in fact, we then have
  $\pf(v_1\wedge\cdots\wedge v_{\ell},v_1\wedge\cdots\wedge v_\ell ) =
  0 $ %
  for each basic $\ell$-vector $v_1\wedge\cdots\wedge v_\ell$. %
\end{ndef}

\begin{rema}\label{KleinQuadric}
  For $n=4$ and $\ell=2$ we are dealing with the space $\Ext[2]{\F^4}$
  that carries the Klein quadric.  The \emph{quadratic form}~$\pq$
  defining the Klein quadric is also referred to as a Pfaffian form
  (cf.~\cite{MR2926161} and~\cite{MR2431124} where this form is
  denoted by~$q$), and $\pf$ is the symmetric bilinear form
  obtained as polarization of that quadratic form~$\pq$.  This is no
  source of confusion as long as $\Char\F\ne2$.  However, the polar
  form~$\pf$ carries less information than the quadratic form~$\pq$ if
  $\Char\F=2$.

  If one interprets the elements of $\Ext[2]{\F^4}$ as alternating
  matrices then there exists a scalar $s\in\F^\myTimes$ such that
  $\pq(X)^2=s\,\det{X}$ holds for each $X\in\Ext[2]{\F^4}$,
  cf.~\cite[\S\,5 no.\,2, Prop.\,2, p.\,84]{MR0107661}; the scalar~$s$
  reflects the choice of basis underlying that interpretation. 
  See~\cite[12.14]{MR1189139} for an interpretation of $\pq$ in
  terms of the exterior algebra. 
\end{rema}

\begin{ndef}[The Hodge operator]
We now consider the composite
\[
  J \coloneqq \pf^{-1}\circ\Ext{h} \colon
\xymatrix@C3em
{%
  \Ext[\ell]V
  \ar[r]^{\Ext{h}}_{\cong}
  & \Ext[\ell]V^\vee
  \ar[r]_{\cong}^{\pf^{-1}}
  & \Ext[n-\ell]V \,.
}
\]
This $\sigma$-linear isomorphism is the \emph{Hodge operator}. It
depends, of course, on~$h$ and on~$b$ but not on the choice of basis.
See, e.g., \cite[pp.\,21--31]{MR0401796} for a discussion of the Hodge
operator for euclidean spaces over~$\RR$.
\end{ndef}

\begin{ndef}[Explicit computation]\label{computeHodge}  
  Suppose that $v_1,\ldots,v_n$ is an orthogonal\footnote{ \ %
    If $\Char\F=2$ and $\sigma=\id$ then an orthogonal basis exists by
    our diagonalizability
    assumption.} %
  basis of~$V$. %
  For $\Ext[\ell]{V}$ we use the basis vectors
  $v_{i_1}\wedge\dots\wedge v_{i_\ell}$ with ascending
  $i_1<\dots< i_\ell\le n$. %
  Then %
  ${\Ext[\ell]{h}(v_1\wedge\cdots\wedge v_\ell,-)}$ %
  is a linear form on $\Ext[\ell]{V}$ which annihilates each one of
  those basis vectors, %
  except for $v_1\wedge\cdots\wedge v_\ell$; in fact
  \[
    \textstyle%
    \Ext[\ell]{h}(v_1\wedge\cdots\wedge v_\ell ,v_1\wedge\cdots\wedge v_\ell ) %
    = h(v_1,v_1)\cdots h(v_\ell ,v_\ell ) \,.
  \]
  In other words: $\Ext[\ell]{h}$ is again diagonalizable, and so
  is~$\Ext{h}$. %
  It then also follows that both
  $\Ext[\ell]h \colon\Ext[\ell]{V}\times\Ext[\ell]{V}\to \F$ and
  $\Ext{h} \colon\Ext{V}\times\Ext{V}\to \F$ are not degenerate. %
  The linear form $\pf(v_{\ell +1}\wedge\cdots\wedge v_n)$ annihilates
  the same collection of basis $\ell $-vectors, and
  \[
    \pf(v_{\ell +1}\wedge\cdots\wedge v_n,v_1\wedge\cdots\wedge v_\ell ) %
    = (-1)^{(n-\ell )\ell} \, b(v_1\wedge\cdots\wedge v_n) \,.
  \]
  Therefore
\begin{eqnarray*}
J(v_1\wedge\cdots\wedge v_\ell ) &=&
v_{\ell +1}\wedge\cdots\wedge v_n \,
\frac{h(v_1,v_1)\cdots h(v_\ell ,v_\ell )}%
{b(v_1\wedge\cdots\wedge v_n)}\, %
(-1)^{(n-\ell )\ell} \,. %
\\
\noalign{\noindent Similarly, we compute}
  J (v_{\ell +1}\wedge\cdots\wedge v_n) &=&
  v_1\wedge\cdots\wedge v_\ell \,
    \frac{h(v_{\ell+1},v_{\ell+1})\cdots h(v_n ,v_n )}%
  {b(v_{\ell +1}\wedge\cdots\wedge v_n\wedge v_1\wedge\cdots\wedge
    v_\ell )} \, %
   (-1)^{(n-\ell )\ell} 
  \\
  &=& 
  v_1\wedge\cdots\wedge v_\ell \,
  \frac{h(v_{\ell+1},v_{\ell+1})\cdots h(v_n ,v_n )}
  {b(v_1\wedge\cdots\wedge v_n)} \,.
\end{eqnarray*}
Note that these formulae are correct only if $v_1,\dots,v_n$ is an
orthogonal basis, and cannot be used if $v_1\wedge\dots\wedge v_\ell$
corresponds to a subspace $U$ of $V$ such that $h|_{U\times U}$ is
degenerate.
\end{ndef}

\begin{ndef}[Functoriality]\label{functoriality}
  It is well known that $\bigotimes^k\colon V\mapsto \bigotimes^kV$
  and $\Ext[k]\colon V\mapsto\Ext[k]{V}$ are functors in the category
  of vector spaces over $\F$ and $\F$-linear maps. We need the fact
  that these are indeed functors in the category of vector spaces over
  $\F$ and $\F$-\emph{semi}linear maps. %
  Since the pertinent facts seem less well known, we take the liberty
  to expand some of the details, in the following.

  Let $V,W$ be vector spaces over~$\F$. An explicit construction of
  the tensor product ${V\otimes W}$ is obtained by factoring the free
  vector space $\F^{(V\times W)}$ with basis ${V\times W}$ over~$\F$
  modulo the subspace $R_{V,W}$ generated by all expressions of the
  form $(v,w)+(v,y)-(v,w+y)$, $(v,w)+(x,w)-(v+x,w)$, $(vc,w)-(v,w)c$
  and $(v,wc)-(v,w)c$ where $v,x\in V$, $w,y\in W$ and $c\in\F$,
  see~\cite[p.\,262]{MR2257570}. %
  As usual, we write $v\otimes w \coloneqq (v,w)+R_{V,W}$.  The tensor
  power $\Tns[k]{V}$ can now be constructed inductively from
  $\Tns[0]{V} \coloneqq \F$ and $\Tns[k+1]{V} \coloneqq
  V\otimes\Tns[k]{V}$.

  In order to extend the morphism part of $\bigotimes^k$ and $\Ext[k]$
  to the category of vector spaces over~$\F$ and
  $\F$-\emph{semi}linear maps we first consider semilinear maps
  $\lambda\colon V\to W$ and $\mu\colon {X\to Y}$ with the same
  companion field homomorphism ${\varphi\colon F\to F}$. %
  Mapping $(v,x)\in V\times X$ to $(\lambda(v),\mu(x))$ extends to an
  $\F$-$\varphi$-semilinear map from~$\F^{(V\times X)}$
  to~$\F^{(W\times Y)}$ which maps $R_{V,X}$ into $R_{W,Y}$. Thus
  $(\lambda\otimes\mu)(v\otimes x) \coloneqq \lambda(v)\otimes\mu(x)$
  defines an $\F$-$\varphi$-semilinear map
  ${\lambda\otimes\mu}\colon {V\otimes X}\to {W\otimes Y}$.

  Proceeding inductively with
  $\mu=\Tns[k-1]{\lambda}\colon\Tns[k-1]{V}\to\Tns[k-1]{W}$, we obtain
  $\F$-$\varphi$-semilinear maps
  \[
  \Tns[k]{\lambda}\colon\Tns[k]{V}\to\Tns[k]{W}\colon
  v_1\otimes\dots\otimes v_k \mapsto
  \lambda(v_1)\otimes\dots\otimes\lambda(v_k)
  \]
  for each~$k\in\NN$.

  Now consider $\F$-semilinear maps $\lambda\colon V\to W$ and
  $\mu\colon W\to Y$ with (possibly different) companion field
  homomorphisms~$\varphi$ and~$\psi$, respectively.  Then
  \[
  \Tns[k]{\mu}\left(
    \Tns[k]{\lambda}(v_1\otimes\dots\otimes v_k)\right)
  = \mu(\lambda(v_1))\otimes\dots\otimes\mu(\lambda(v_k)) =
  \Tns[k]{(\mu\circ\lambda)}(v_1\otimes\dots\otimes v_k)
  \]
  shows $\Tns[k]{\mu}\circ\Tns[k]{\lambda} =
  \Tns[k]{(\mu\circ\lambda)}$, and functoriality of $\Tns[k]$ is
  established.  Combination of these functors yields a functor $\Tns$
  with $\Tns{V} \coloneqq \bigoplus_{k\in\NN}\Tns[k]{V}$ and
  $(\Tns{\lambda})(\sum_{k\in\NN}t_k)  \coloneqq 
  \sum_{k\in\NN}\Tns[k]{\lambda}(t_k)$.

  Finally we consider the exterior algebra
  $\Ext{V}=\bigoplus_{k\in\NN}\Ext[k]{V}$ which is obtained as the
  quotient of $\Tns{V}$ modulo the two-sided ideal $S_V$ generated by
  $\set{v\otimes v}{v\in V}$, cf.~\cite[p.\,288]{MR2257570}. We write
  $v_1\wedge\dots\wedge v_k \coloneqq %
  (v_1\otimes\dots\otimes v_k)+S_V$ as usual. For each $\F$-semilinear
  map $\lambda\colon V\to W$ we note that $\Tns{\lambda}(S_V)$ is
  contained in~$S_W$. Thus
  $\Ext[k]{\lambda}(v_1\wedge\dots\wedge v_k) \coloneqq
  \lambda(v_1)\wedge\dots\wedge\lambda(v_k)$ is well defined.  Note
  that we use the multiplication in $\Tns{V}$ only for the definition
  of~$S_V$.
\end{ndef}

\begin{exam}\label{muDet}
  Assume $\dim{V}=n$ and consider an $\F$-semilinear endomorphism
  ${\lambda\colon V\to V}$ with companion field
  endomorphism~$\varphi$.  Choose a basis $v_1,\dots,v_n$ for~$V$;
  then $\lambda$ is the composite $\lambda=\lambda'\circ\hat\varphi$
  of some linear map $\lambda'\colon V\to V$ with the map
  $\hat\varphi\colon\sum_{k=1}^n v_k x_k \mapsto \sum_{k=1}^n v_k
  \varphi(x_k)$. It is well known (cf.~\cite[Thm.\,4.11]{MR1878556}
  or~\cite[III\,\S\,8]{MR979982}) %
  that $\Ext[n]\lambda'$ is just multiplication by
  $\det{\lambda'}$. For each $s\in\F$ we now have
  $\left(\Ext[n]{\hat\varphi}\right)(v_1s\wedge\dots\wedge v_n) =
  {(v_1\wedge\dots\wedge v_n)\varphi(s)}$ and thus
  \[
    \begin{array}{rcl}
      \left(\Ext[n]{\lambda}\right)\left((v_1\wedge\dots\wedge v_n)s\right) &=&
      {(\Ext[n]{\lambda'}\circ
        \Ext[n]{\hat\varphi})\left(v_1s\wedge\dots\wedge v_n\right)}
      \\
   &=& \left(v_1\wedge\dots\wedge v_n\right) \,\varphi(s)\, (\det\lambda')\,.
    \end{array}
  \]
\end{exam}

\begin{nthm}[The unitary group]\label{semiSimilitudesAct}
  The special unitary group $\SUV{V,h}$ acts in a natural way both
  on~$\Ext{V}$
  and on~$\Ext{V^\vee}$
  and commutes with $\Ext{h}\colon \Ext{V}\to\Ext{V}^\vee$
  because~$\Ext$ is a functor. %
  The group $\SUV{V,h}$ also commutes with
  $\pf\colon \Ext[n-\ell]V\to\Ext[\ell]V^\vee$ because its elements
  have determinant~$1$. Therefore it centralizes the Hodge
  operator~$J$.  %
  More generally, the groups
  \[
  \begin{array}{rcl}
    \GUV{V,h} &\coloneqq& \set{\lambda\in\GL{V}}{\exists\,
      r_\lambda\in\F\,\forall\, v,w\in V\colon
      h\left(\lambda(v),\lambda(w)\right) = r_\lambda \, h(v,w) } \text{ and}\\[1ex]
    \gUV{V,h} &\coloneqq& \set{\lambda\in\gL{V}}{\exists\,
      r_\lambda\in\F\,\forall\, v,w\in V\colon
      h\left(\lambda(v),\lambda(w)\right) = r_\lambda \,\varphi_\lambda(h(v,w)) }
  \end{array}
  \]
  of similitudes and semi-similitudes, respectively, also act
  naturally on $\Ext{V}$; here $\varphi_\lambda$ denotes the companion
  field automorphism of $\lambda\in\gL{V}$. %
  A general element~$\lambda$ of $\gUV{V,h}$ will not centralize the
  Hodge operator but it will normalize the set~$\F\id\circ J$; %
  in fact, the conjugate
  $\Ext[n-\ell]\lambda\circ J\circ(\Ext[\ell]\lambda)^{-1} = t_\lambda\id\circ
  J$, with
  \[
    t_\lambda =
    \frac
    {\varphi_\lambda(\det\lambda')}
    {r_\lambda^{\ell}}
    \, %
    \frac{b({v_1\wedge\dots\wedge v_n})}%
    {\varphi_\lambda(b({v_1\wedge\dots\wedge v_n}))} \,  %
  \]
  where $v_1,\dots,v_n$ is an orthonormal basis used to obtain a
  decomposition $\lambda = \lambda'\circ\hat\varphi$ as
  in~\ref{muDet}.
  (This simplifies to $t_\lambda = r_\lambda^{-\ell}\,\det\lambda$
  if $\lambda$ is $\F$-linear.)
\end{nthm}
\begin{proof}
  Only the assertion about
  $\Ext[n-\ell]\lambda\circ J\circ(\Ext[\ell]\lambda)^{-1}$ remains to
  be verified. %
  Let $v_1,\dots,v_n$ be the orthogonal basis that was used for the
  decomposition $\lambda = \lambda'\circ\hat\varphi$ as
  in~\ref{muDet}, and put $w_i \coloneqq \lambda(v_i)$. Then
  $w_1,\dots,w_n$ is an orthogonal basis, and
  $(\Ext[\ell]\lambda)^{-1}(w_1\wedge\dots\wedge w_\ell) =
  v_1\wedge\dots\wedge v_\ell$. %
  Using~\ref{computeHodge}, we obtain that %
  $\Ext[n-\ell]\lambda\circ J\circ(\Ext[\ell]\lambda)^{-1}$ maps %
  ${w_1\wedge\dots\wedge w_\ell}$ to %
  \[
    \begin{array}{rcr}
      \left(\Ext[n-\ell]\lambda\circ J\right)
      ({v_1\wedge\dots\wedge v_\ell}) %
      &=&
          \left(\Ext[n-\ell]\lambda\right) %
          \left( %
          v_{\ell+1}\wedge\dots\wedge v_n \, %
          \frac{h(v_1,v_1)\cdots h(v_\ell ,v_\ell )}%
          {b(v_1\wedge\cdots\wedge v_n)}\, %
          (-1)^{(n-\ell )\ell} %
          \right)
      \\
      &=&
          (w_{\ell+1}\wedge\dots\wedge w_n) %
          \quad\varphi_\lambda\left(\frac{ h(v_1,v_1)\cdots h(v_\ell ,v_\ell )}%
          {b(v_1\wedge\cdots\wedge v_n)}\right)\, %
          (-1)^{(n-\ell )\ell} \,.            
    \end{array}
  \]
  Now %
  \( %
  h(w_i,w_i) = h(\lambda(v_i),\lambda(v_i)) =
  r_\lambda\varphi_\lambda(h(v_i,v_i)) %
  \) %
  yields \( %
  \varphi_\lambda(h(v_i,v_i)) = r_\lambda^{-1}h(w_i,w_i) %
  \). %
  From~\ref{muDet} we know
  ${w_1\wedge\dots\wedge w_n} = (\Ext[n]{\lambda})
  ({v_1\wedge\cdots\wedge v_n}) = {v_1\wedge\dots\wedge
    v_n}\det{\lambda'}$.  So
  $\varphi_\lambda\bigl(b({v_1\wedge\cdots\wedge v_n})\bigr) =
  \varphi_\lambda\bigl(b({w_1\wedge\cdots\wedge
    w_n})(\det\lambda')^{-1}\bigr)$, and the assertion follows.
\end{proof}

\begin{ndef}[Application in 3D]
  Let $h$ be the standard bilinear form on~$\RR^3$; %
  i.e., a form with an ortho\emph{normal} basis
  $v_1,v_2,v_3$. Moreover, let %
  $b(v_1\wedge v_2\wedge v_3)=-1$. Then $J(v_1)=-v_2\wedge v_3$,
  $J(v_2)=v_1\wedge v_3$, $J(v_3)=-v_1\wedge v_2$,
  see~\ref{computeHodge}. In other words:
  \[
  J\left(
    \begin{matrix}
      x\\
      y\\
      z
    \end{matrix}\right)
  =
  \left(
    \begin{matrix}
      0 & -z & y \\
      z & 0 & -x \\
      -y & x & 0
    \end{matrix}\right)
  \]
  and this map is $\SO{3}{\RR}$-equivariant, while
  $\lambda\circ J\circ\lambda^{-1} = -J$ holds for each
  $\lambda\in\Orth{3}{\RR}\smallsetminus\SO{3}{\RR}$.

  Let
  $\gamma \colon\Omega\to\RR^3\colon (x,y,z)^\transp \mapsto (\gamma
  _1(x,y,z),\gamma _2(x,y,z),\gamma _3(x,y,z))^\transp$ be a vector
  field defined on some domain $\Omega\subseteq\RR^3$. %
  Then the Hodge operator maps $\operatorname{curl}\gamma = ({ %
    \partial_y\gamma _3 - \partial_z\gamma _2 , %
    \partial_z\gamma _1 - \partial_x\gamma _3 , %
    \partial_x\gamma _2 - \partial_y\gamma _1 %
  })^\transp $ %
  to the anti-symmetrization of the Jacobian
  $\operatorname{Jac}\gamma$; indeed
  \[
  \operatorname{Jac}\gamma -(\operatorname{Jac}\gamma )^\transp =
  \left(
    \begin{matrix}
      0 & \partial_y\gamma _1-\partial_x\gamma _2 & \partial_z\gamma _1-\partial_x\gamma _3 \\
      \partial_x\gamma _2-\partial_y\gamma _1 & 0 & \partial_z\gamma _2-\partial_y\gamma _3 \\
      \partial_x\gamma _3-\partial_z\gamma _1 & \partial_y\gamma _3-\partial_z\gamma _2 & 0
    \end{matrix}
  \right) = J\left(
    \begin{matrix}
      \partial_y\gamma _3 - \partial_z\gamma _2 \\
      \partial_z\gamma _1 - \partial_x\gamma _3 \\
      \partial_x\gamma _2 - \partial_y\gamma _1 
    \end{matrix}
  \right) \,.
  \]
\end{ndef}

\begin{nthm}[The square of the Hodge operator]\label{squareHodge}
  Let $H$ be the Gram matrix of~$h$ with respect to the basis
  $v_1,\dots,v_n$. The square of~$J$ is a linear automorphism
  of~$\Ext[\ell]V$, and we find that
  \[\textstyle
  J^2 = \delta_\ell\, \id \quad\text{ where }\quad
  \delta_\ell \coloneqq (-1)^{(n-\ell )\ell}
  \frac{\det(H)}{N_{\F|\R}(b(v_1\wedge\cdots\wedge v_n))} \,.
  \]
\end{nthm}
Recall that $\det(H)$ depends on the choice of basis; the invariant
would be $\disc(h)\in\F^\myTimes/N_{\F|\R}(\F^\myTimes)$, the norm
class of $\det(h(v_i,v_j))$. %
However, the whole expression depends only on $h$ and~$b$. Replacing
the isomorphism $b\colon \Ext[n]V\to\F$ changes~$J$ by a factor
and~$J^2$ by the norm of that factor. In particular, the isomorphism
type of the algebra $\K_\ell$ introduced in~\ref{def:algebraK} below
does not depend on the choice of~$b$.  If $b$ is chosen in such a way
that $b(v_1\wedge\cdots\wedge v_n)=1$, then we arrive at the textbook
formula $ J^2=(-1)^{(n-\ell )\ell}\det(H)\,\id$, cp.
\cite[6.1]{MR722297}.

\goodbreak%
\begin{lemm}\label{HodgeSemiSimilitude}
  For all $x,y\in\Ext[\ell]V$ we have 
    \begin{enumerate}
    \item\label{firstAss}%
      $\pf(J(x),y) = \Ext[\ell]{h}(x,y)$, %
    \item $\pf(x,J(y)) =
      (-1)^{(n-\ell)\ell}\,\gal{\Ext[\ell]{h}(x,y)}$, %
    \item $\pf(J(x),J(y)) = \delta_\ell\,\gal{\pf(y,x)}$, %
    \item\label{secondAss}%
      $\Ext[\ell]{h}(J(x),y) = \delta_\ell\,{\pf(x,y)}$, %
    \item $\Ext[\ell]{h}(x,J(y)) = \delta_\ell\,\gal{\pf(y,x)}$, %
    \item\label{skewH}%
      $\Ext[\ell]{h}(J(x),J(y)) = %
      (-1)^{(n-\ell)\ell}
      \,\delta_\ell\,\gal{\Ext[\ell]{h}(x,y)}$. %
    \end{enumerate}
\end{lemm}
\begin{proof}
  We compute
  \(%
  \pf(J(x),-) = (\pf\circ\pf^{-1}\circ\Ext[\ell]{h})(x,-) =
  \Ext[\ell]{h}(x,-) %
  \) and 
  \(%
  \Ext[\ell]{h}(J(x),-) =
  (\Ext[\ell]{h}\circ\pf^{-1}\circ\Ext[\ell]{h})(x,-) =
  (\pf\circ\pf^{-1}\circ\Ext[\ell]{h}\circ\pf^{-1}\circ\Ext[\ell]{h})(x,-)
  = \pf(J^2(x),-) = {\delta_\ell\,\pf(x,-)} %
  \).
  This gives the
  assertions~\ref{firstAss} and~\ref{secondAss}, the rest follows from
  $\Ext[\ell]{h}(y,x)=\gal{\Ext[\ell]{h}(x,y)}$,
  $\pf(y,x)=(-1)^{(n-\ell)\ell}\,\pf(x,y)$ and $J^2=\delta_\ell\,\id$.
\end{proof}

\goodbreak
\section{Composition algebras generated by Hodge operators}
\label{sec:modulesQuat}

From now on, assume $n=2\ell$.  Then $J$ gives an
$\F$-$\sigma$-semilinear endomorphism
\[\textstyle
J\colon \Ext[\ell]V\to\Ext[\ell]V.
\]
We are going to use $J$ to give $\Ext[\ell]{V}$ the structure of a right
module over an associative algebra of dimension $2\,\dim_\R\F$ over~$\R$.
Note that $(-1)^{(n-\ell)\ell} = (-1)^\ell$ holds because $n=2\ell$.

\begin{ndef}[The algebra $\noexpand\K_\ell$]\label{def:algebraK}
  Let
  $\delta_\ell \coloneqq
  (-1)^{\ell}\frac{\det(H)}{N_{\F|\R}(b(v_1\wedge\cdots\wedge v_n))}$
  as in~\ref{squareHodge} and let $\K_\ell$ denote the $\R$-algebra
  consisting of all matrices of the form %
  \[\textstyle
    x = %
    \begin{pmatrix} x_0 & \delta_\ell\,\gal{x_1}\\
      x_1 & \gal{x_0}
    \end{pmatrix}
    \in\F^{2\times 2}.
  \]
  We identify $x_0\in \F$ with the diagonal matrix
  $\begin{pmatrix}x_0&0\\
    0&\gal{x_0}\end{pmatrix}$ and put
  $\bj_\ell^{} \coloneqq \begin{pmatrix}0&\delta_\ell
    \\1&0 \end{pmatrix}$.  Thus %
 \[
 \begin{pmatrix}
   x_0& \delta_\ell\,\gal{x_1}\\
   x_1&\gal{x_0}
 \end{pmatrix}=x_0+ \bj_\ell^{}\,x_1 \,.
 \]
 We also have $\bj_\ell^2=\delta_\ell$ and
 $\bj_\ell^{}\, x=\gal{x}\,\bj_\ell^{}$ for each $x\in \F$.  %
 Note $\dim_\R\K_\ell = 2\,\dim_\R\F$.  For $x_0,x_1\in\F$ we put
 $\kappa(x_0+\bj_\ell\, x_1) \coloneqq \gal{x_0}-\bj_\ell^{}\,x_1$.  %
 The map $\kappa$ is called the \emph{standard involution}
 of~$\K_\ell$. %
 We also write $\gal{x} \coloneqq \kappa(x)$ for $x\in\K_\ell$. %
 As~$\kappa$ extends the action of~$\sigma$ on $\F\cong\set{\left(
     \begin{smallmatrix}
       x_0 & 0 \\
       0 & \gal{x_0}
     \end{smallmatrix}\right)}{x_0\in\F}$, %
 this should cause no confusion.
\end{ndef}

We write $\F^\myTimes\coloneqq \F\smallsetminus\{0\}$ for the
multiplicative group, %
$\F^\sq \coloneqq \set{s^2}{s\in\F}$, and
$\F^\sqt \coloneqq \F^\sq\smallsetminus\{0\}$.  Then
$\F^\myTimes/\F^\sqt$ is the group of non-trivial square classes
of~$\F$; note that this is an elementary abelian $2$-group.  %

\begin{lemm}
  The standard involution is an anti-automorphism of~$\K_\ell$
  satisfying %
  $\kappa(x)\,x=x\,\kappa(x) = \det_\F\left( x_0+\bj_\ell\,x_1 \right)
  = N_{\F|\R}(x_0)- \delta_\ell \, N_{\F|\R}(x_1) \in\R$ for every
  $x = {x_0+\bj_\ell\,x_1} \in \K_\ell$.  %
  Thus the determinant gives a multiplicative quadratic form
  $\det\colon\K_\ell\to\R$.  The corresponding polarization $f_{\det}$
  is degenerate precisely if ${\Char\F=2}$ and\/~$\sigma=\id$; %
  indeed, it vanishes in that case. %
  The quadratic form $\det$ is then degenerate if, and only if, the
  bilinear form~$h$ has discriminant\/~$1\in\F^\myTimes/\F^\sqt$.  %
\end{lemm}
\begin{proof}
  Standard matrix computations suffice to verify that $\kappa$ is
  indeed an anti-automorphism, with the properties as claimed. %
  The value of the polar form at
  $(x_0+\bj_\ell\,x_1,y_0+\bj_\ell\,y_1)$ is obtained as %
  \( %
  \det_\F\left(x_0+y_0+\bj_\ell(x_1+y_1)\right)
  -\det_\F\left(x_0+\bj_\ell\,x_1\right)
  -\det_\F\left(y_0+\bj_\ell\,y_1\right) =
  x_0\gal{y_0}+y_0\gal{x_0}-\delta_\ell(x_1\gal{y_1}+y_1\gal{x_1})
  \). %
  Clearly, this polar form is zero if ${\Char\F=2}$
  and\/~$\sigma=\id$. %

  If $\sigma\ne\id$ we pick $s\in\F\smallsetminus\R$ with $s+\gal{s}=1$;
  then $s\gal{s}=s-s^2$. With respect to the basis $(1,0)$, $(s,0)$,
  $(0,1)$, $(0,s)$, the Gram matrix of the polar form becomes
  \[
    \left(
      \begin{smallmatrix}
        2 & 1 & 0 & 0 \\
        1 & 2s\gal{s} & 0 & 0 \\
        0 & 0 & -2\delta_\ell & -\delta_\ell \\
        0 & 0 & -\delta_\ell & -2\delta_\ell s\gal{s} \\
      \end{smallmatrix}
    \right),
  \]
  with determinant
  $\delta_\ell^2(4s\gal{s}-1)^2 = \delta_\ell^2(1-2s)^2 \ne0$.

  If $\sigma=\id$ then the Gram matrix $\left(
    \begin{smallmatrix}
      2 & 0 \\
      0 & -2\delta_\ell
    \end{smallmatrix}
  \right)$ for the polar form is nonsingular precisely if
  $\Char\F\ne2$.

  Finally, assume $\sigma=\id$ and $\Char\F\ne2$. Then~$h$ has
  discriminant~$1$ precisely if~$\delta_\ell$ is a square.  In that
  case, the quadratic form
  $\det_\F(x_0+\bj_\ell\,x_1) = x_0^2-\delta_\ell x_1^2 =
  (x_0-\sqrt{\delta_\ell}x_1)^2$ is isotropic, and then degenerate
  because its polar form is zero.
\end{proof}

\begin{rems}\label{rems:compositionalgebra}
  The isomorphism type of~$\K_\ell$ as an $\R$-algebra does not depend
  on the choice of the orthogonal basis $v_1,\dots,v_n$, and does not
  depend on the choice of~$b$; cp. the discussion
  following~\ref{squareHodge}. %
  In the sequel, we will choose the basis in such a way that the Gram
  matrix~$H$ has a convenient determinant, and then usually adapt~$b$
  such that $b(v_1\wedge\cdots\wedge v_n) = 1$; %
  then $\delta_\ell = (-1)^\ell \det H$. %

  If the polar form $f_{\det}$ is non-degenerate then $\K_\ell$ is
  indeed a \emph{composition algebra}, cf.~\cite[1.5.1]{MR1763974}
  or~\cite[\S\,7.6]{MR780184}. In this case the standard
  involution~$\kappa$ is uniquely determined by the properties of
  being an $\R$-linear anti-automorphism of~$\K_\ell$ and
  inducing~$-\id$ on $\R^\perp$. %
  As such, it is rightfully termed ``standard''. %
  Note that
  $\R^\perp =
  \smallset{x\in\K_\ell\smallsetminus\R}{x^2\in\R}\cup\{0\}$ if
  $\Char\R\ne2$, and $\R^\perp = \smallset{x\in\K_\ell}{x^2\in\R}$
  if~$\Char\R=2$. %
  The norm form of the composition algebra $\K_\ell$ is a Pfister
  form; it is the orthogonal sum of the norm form $N_{\F|\R}$ and its
  scalar multiple $-\delta_\ell N_{\F|\R}$.

  If $f_{\det}$ is degenerate we have $\kappa = \id$ and
  $\det(x)=x\gal{x}=x^2$. These inseparable cases will be discussed
  in~\cite{KramerStroppel-hodge-char2-arXiv} in greater detail. %

  Note that~$\K_\ell$ contains non-invertible elements different
  from~$0$ if, and only if, the factor~$\delta_\ell$ is a norm; i.e.,
  if $(-1)^\ell$ lies in $N_{\F|\R}(\F)$. In that case, we call the
  algebra~$\K_\ell$ \emph{split}; we may (and will) normalize
  $\delta_\ell=1$ then. %
  The case where~$\K_\ell$ is split will be discussed in
  Section~\ref{sec:SplitCases} below.  In particular, the existence of
  idempotents or of nilpotent elements in
  $\K_\ell\smallsetminus\{0,1\}$ leads to the existence of certain
  $\SUV{V,h}$-submodules in $\Ext[\ell]V$,
  see~\ref{splitCaseSplitModule} below. %

\enlargethispage{5mm}%
  If $\delta_\ell$ is not a norm then $\K_\ell|\F$ is either a
  quadratic extension (for $\sigma=\id$; that extension is inseparable
  if $\Char \F=2$) or a quaternion division algebra over $\R$ (for
  $\sigma\neq \id$).  See Section~\ref{sec:nonSplit} below.
\end{rems}

\goodbreak%
\begin{prop}\label{similitudeAlgebraAut}
  Via conjugation in $\End[\R]{\Ext[\ell]V}$, every similitude
  of\/~$h$ induces an automorphism of the $\R$-subalgebra~$M_\ell$ %
  generated by~$J$ and~$\F\,\id$ in $\End[\R]{\Ext[\ell]V}$.
\end{prop}
\begin{proof}
  Let~$\gamma$ be a similitude with multiplier~$r$. Then
  $\gamma\circ J\circ\gamma^{-1} = Jr^{-\ell}\det\gamma$.  Using a
  matrix $A$ describing~$\gamma$ and the Gram matrix~$H$ for~$h$ with
  respect to the same basis, we note that $\gal{A}\transp H A = rH$,
  and infer $N_{\F|\R}(\det A)=r^{2\ell}$. So
  $N_{\F|\R}(r^{-\ell}\det\gamma) = 1$, and a straightforward
  calculation yields that conjugation by~$\gamma$ yields an algebra
  automorphism, as claimed.
\end{proof}

Obviously, the algebra~$M_\ell$ is anti-isomorphic to~$\K_\ell$, and
then in fact isomorphic because~$\K_\ell$ admits an anti-automorphism.

\begin{defi}\label{def:Kmodule}
  For $v\in\Ext[\ell]V$ we put
  $  v\,\bj_\ell^{} \coloneqq J(v)$.

  In this way, $\Ext[\ell]V$ becomes a $\SUV{V,h}$-$\K_\ell$-bimodule,
  i.e., it becomes a right module over~$\K_\ell$ and $\SUV{V,h}$ acts
  $\K_\ell$-linearly from the left. %
  Choosing an orthogonal basis $v_1,\dots,v_n$ for~$V$ with a fixed
  ordering we obtain a basis~$B$ for $\Ext[\ell]V$ consisting of all
  $v_{j_1}\wedge\dots\wedge v_{j_\ell}$ where $(j_1,\dots,j_\ell)$ is
  an increasing sequence of length~$\ell$ in $\{1,\dots,n\}$.  The
  sequences with $j_1=1$ form a subset~$B_1$ of~$B$, and~$J$ maps each
  element of~$B_1$ to one of $B\smallsetminus B_1$. Moreover, the
  set~$B_1$ %
  forms a basis for the $\K_\ell$-module $\Ext[\ell]V$, showing that
  the latter is a free module.
\end{defi}

\begin{ndef}[The $\alpha$-hermitian form]\label{def:g}
  We define $g\colon\Ext[\ell]{V}\times\Ext[\ell]{V}\to\K_\ell$ by
\[\textstyle
g(u,v) \coloneqq \Ext[\ell]{h}(u,v)+\Ext[\ell]{h}(u,v\bj_\ell^{} %
)\,\bj_\ell^{-1}
= \Ext[\ell]{h}(u,v)+\bj_\ell\,(-1)^\ell\,\pf(u,v)
\,;
\]
see~\ref{HodgeSemiSimilitude} for the description on the right hand
side.  This expression is $\F$-linear in the right argument (i.e.,
$g(u,vs)=g(u,v)s$), and $\F$-$\sigma$-semilinear in the left argument
(in the sense that $g(us,v)=\gal{s}g(u,v)$ for all $s\in \F$ and all
${u,v\in\Ext[\ell]{V}}$). %
Moreover, %
\(%
g(u,v\bj_\ell) = \Ext[\ell]{h}(u,v\bj_\ell) +
\Ext[\ell]{h}(u,(v\bj_\ell)\bj_\ell)\,\bj_\ell^{-1} %
= %
\Ext[\ell]{h}(u,v\bj_\ell) +
\Ext[\ell]{h}(u,v)\,\delta_\ell\,\bj_\ell^{-1} %
g(u,v)\,\bj_\ell^{} %
\), %
so $g$ is indeed $\K_\ell$-linear in the right argument.  Let $\alpha$
denote the following involution: %
\[
\alpha(x_0+ \bj_\ell^{}\,x_1) \coloneqq %
\gal{x_0} + (-1)^\ell \,\bj_\ell^{}\,x_1 \,.
\]
Then~$\alpha$ is an $\R$-linear anti-automorphism of~$\K_\ell$, with
$\alpha(x) = \gal{x}$ for $x\in\F$ and $\alpha(\bj_\ell^{}) =
(-1)^\ell\,\bj_\ell^{}$.  %
Using the fact that $\pf$ is graded symmetric we find $g(v,u) =
\alpha(g(u,v))$.  This shows that $g$ is an $\alpha$-hermitian form on
$\Ext[\ell]V$.
\end{ndef}

Note that $\K_\ell$ is not a field, in general: we need the more
general concept of hermitian forms over rings (e.g.,
see~\cite{MR1096299}).

\begin{prop}\label{gDiagonal}
  The form $g$ is diagonalizable. %
\end{prop}
\begin{proof}
  In fact, for any orthogonal basis $v_1,\dots,v_n$ of~$V$, consider
  the vector %
  $w_\gamma \coloneqq v_1\wedge v_{\gamma_2}\wedge\dots\wedge
  v_{\gamma_\ell}$ for each increasing sequence
  $\gamma = (\gamma_2,\dots,\gamma_\ell)$. Then the $w_\gamma$ form a
  basis for the free $\K_\ell$-module
  $\Ext[\ell]V$. By~\ref{computeHodge}, we have
  $\Ext[\ell]h(w_\gamma,w_\beta) = 0$ if $\gamma\ne\beta$, and
  $\Ext[\ell]h(w_\gamma,w_\gamma) =
  h(v_1,v_1)h(v_{\gamma_2},v_{\gamma_2})\dots
  h(v_{\gamma_\ell},v_{\gamma_\ell})$ for each admissible~$\gamma$.

  As every $w_\gamma$ is a wedge product starting with the same
  vector~$v_1$, we have $\pf(w_\gamma,w_\beta) = 0$, and
  $g(w_\gamma,w_\beta) = \Ext[\ell]h(w_\gamma,w_\beta)$.
 \end{proof}

\begin{rems}\label{alphaVSkappa}
  The involution $\alpha$ coincides with the standard involution
  $\kappa$ if $\Char\F=2$ or if~$\ell$ is odd.  If $\sigma\ne\id$ and
  $\ell$ is even we pick any $a\in\F^\myTimes$ with $\gal{a}=-a$ and
  obtain $\alpha(x) = a^{-1}\kappa(x)a$ for all $x\in\K_\ell$. %
  If $\sigma=\id$ and $\ell$ is even then $\alpha=\id$. %
  The idea of (re-)constructing the hermitian form~$g$ dates back
  to~\cite[p.\,266]{MR0001957}, cf.~\cite[4.4]{MR2942723}
  and  ~\cite[2.9]{MR3871471}. %
\end{rems}

\begin{lemm}\label{isotropic}\label{gAnisotropicIfhIs}
  \begin{enumerate}
  \item For $X,Y\in\Ext[\ell]{V}$ we have
    \[
      g(X,Y)=0 \iff \pf(X,Y)=0=\Ext[\ell]{h}(X,Y) \,.
    \]
    In particular, the form~$g$ is not degenerate (since~$h$ is not
    degenerate, see~\ref{computeHodge}). %
  \item
    If\/ $\Ext[\ell]{h}$ is anisotropic then $g$ is anisotropic,
    too. %
  \item If\/ $h$ is isotropic then both~$\Ext[\ell]h$ and~$g$ are
    isotropic, as well.
  \end{enumerate}
\end{lemm}

\begin{proof}
  Comparing coefficients in $\K_\ell=\F\oplus\bj_\ell\F$ we obtain the
  first assertion. The second one follows immediately.

  The form $\Ext[1]h$ coincides with~$h$. For $n =\dim{V}$ the
  form~$\Ext[n]{h}$ on $\Ext[n]V \cong \F^1$ will be anisotropic,
  regardless of what the (non-degenerate) form~$h$ may be. %
    
  So consider $2\le\ell<n$, and assume that $h$ is isotropic. %
  Then~$V$ is the orthogonal sum of an anisotropic space and a
  non-trivial split space~$S$, %
  see~\cite[III, (1.1), p.\,56]{MR0506372}. %
  In the split space~$S$, we find $x,y$ such that $h(x,x)=0$ and
  $h(x,y)=1$, see~\cite[I, (6.3), p.\,56]{MR0506372}. %

  If $\Char\F\ne2$, we may also assume $h(y,y)=0$. We then put
  $v_1\coloneqq x+y$ and $v_2\coloneqq x-y$, and extend to an
  orthogonal basis $v_1,v_2,\dots,v_n$ for~$V$.

  If $\Char\F=2$, there are two possibilities: either the restriction
  of~$h$ to~$S$ is alternating, %
  or it is possible to choose $x,y\in S$ such that
  $h(x,x)=0\ne h(y,y)$ and $h(x,y)=1$. %
  In the first case, we find $u\in S^\perp\smallsetminus\{0\}$ because the
  (diagonalizable and non-degenerate) form~$h$ is not alternating. %
  In the second case, we just take $u=0$. %
  We now put $v_1\coloneqq y+u$ and $v_2\coloneqq xh(v_1,v-y_1)$. %
  Then $h(v_2,v_2)=h(v_1,v_1)\ne0$, and $h(v_1,v_2)=0$.  So there
  exists an orthogonal basis $v_1,v_2,\dots,v_n$ for~$V$.
    
  If $\ell=2$, %
  put $w_1 \coloneqq v_1\wedge v_3$ %
  and $w_2 \coloneqq v_2\wedge v_3$. %
  If $\ell>2$, %
  put $w_1 \coloneqq v_1\wedge v_3\wedge\dots\wedge v_{\ell+1}$ %
  and $w_2 \coloneqq v_2\wedge v_3\wedge\dots\wedge v_{\ell+1}$. %
  Then $\Ext[\ell]{h}(w_1,w_1) = -\Ext[\ell]{h}(w_2,w_2)$ yields
  $\Ext[\ell]{h}(w_1+w_2,w_1+w_2) = 0$, so $\Ext[\ell]{h}$ is
  isotropic.

  Finally, recall from~\ref{def:Pfaffian} that
  $\pf(w_1+w_2,w_1+w_2) = 0$, so the form~$g$ is isotropic, as well.
\end{proof}

\begin{rems}\label{exthMayBecomeIsotropic}
  It may happen that~$h$ is anisotropic but $\Ext[\ell]{h}$ becomes
  isotropic. For instance, this happens whenever we start with an
  anisotropic bilinear form~$h$ on a vector space of dimension~$4$
  over a $p$-adic field because there are no anisotropic quadratic
  forms in more than four variables over such a field
  (see~\cite[VI\,2.2]{MR2104929}). %
  
  For more explicit examples, take~$h_1$ and~$h_2$ as the
  polarizations of the quadratic forms $x_0^2-2x_1^2+3x_2^2-6x_3^2$
  and $x_0^2+2x_1^2+10x_3^2-5x_4^2$, respectively, over~$\QQ$.  %
  Both forms are anisotropic; the discriminant of~$h_1$ is a square
  while that of~$h_2$ is not. %
  Both~$\Ext[2]{h_1}$ and~$\Ext[2]{h_2}$ are isotropic. %
  For~$h_2$, the corresponding form~$g$ is isotropic (over
  $\K_2\cong\QQ(\sqrt{-1})$), as well,
  see~\ref{exam:anisotropicNonSplit} below.
\end{rems}

\begin{nthm}[The homomorphism]\label{homoSUtoSaU}
  From the definition of~$g$ it is clear that $\SUV{V,h}$
  preserves~$g$ and that $\gUV{V,h}$ acts by semi-similitudes of~$g$,
  see~\ref{semiSimilitudesAct}. %
  Thus we have, for $\dim(V)=n=2\ell$, constructed a homomorphism %
  \( %
  \eta^{}_\ell\colon \gUV{V,h}\to\gUV{\Ext[\ell]V,g} \) %
  which restricts to
  \[
  \eta^{}_\ell|_{\SUV{V,h}}\colon \SUV{V,h}\to\UV{\Ext[\ell]V,g} \,.
  \]
  Of course, we have to read $\gUV{V,h}$ as $\gOV{V,h}$ if
  $\sigma=\id$, and $\gUV{\Ext[\ell]V,g}$ as $\gOV{\Ext[\ell]V,g}$ if
  $\sigma=\id$ and~$\ell$ is even. 
\end{nthm}

\begin{lemm}\label{kernelEta}
  The kernel of $\eta^{}_\ell$ consists of all scalar multiples 
  $s\,\id$ of the identity where $\gal{s}s=1=s^\ell$. %
  The center of the image $\eta^{}_\ell(\SUV{V,h})$ has
  order at most~$2$ because $1=\det(s\,\id)$ already implies
  $s^n=s^{2\ell}=1$. %
\end{lemm}
\begin{proof}
  Via $(v_1\wedge\dots\wedge v_\ell)\F \mapsto v_1\F+\dots+v_\ell\F$
  we identify $\Gr[\F]1{\Ext[\ell]V}$ with $\Gr[\F]\ell{V}$. If some
  $A\in\gL{V}$ with companion~$\mu_A$ fixes each element
  of~$\Gr[\F]\ell{V}$ then it also fixes each element of~$\Gr[\F]1V$
  because these are obtained as intersections of $\ell$-dimensional
  spaces. %
  As~$\F$ is commutative, we find (by %
  evaluating~$A$ on a basis $b_1,\dots, b_n$ and at $b_1+\cdots+b_nt$
  for $t\in\F$, see~\cite[III.1, Prop.\,3(b), p.\,43]{MR0052795}) that
  there exists $s\in\F$ such that $A=s\,\id$ and $ts=s\mu_A(t)$. In
  particular, $A$ is linear. %
  Now~$A$ is trivial on~$\Ext[\ell]V$ precisely if $s^\ell=1$. %
  The rest is clear.
\end{proof}

\begin{ndef}[Example: $\noexpand\SaU{n}{\noexpand\HH}$]\label{def:SaU}
  Let $\HH = \RR+i\RR+j\RR+k\RR$ be the (essentially unique)
  quaternion division algebra over~$\RR$. Then $\CC = \RR+i\RR$ is the
  field of complex numbers, and we obtain $\HH = \CC+j\CC$. Composing
  the standard involution $\kappa\colon a\mapsto\gal{a}$ with the
  inner automorphism $a\mapsto i^{-1}ai = -iai$, we obtain the
  involutory anti-automorphism $\alpha\colon x+jy \mapsto \gal{x}+jy$
  (here $x,y\in\CC$). %
  On $\HH^n$, we now define the $\alpha$-hermitian form $f\colon
  \HH^n\times\HH^n\to\HH$ by $f\bigl((a_1,\dots,a_n),(b_1,\dots b_n)\bigr)
  \coloneqq
  \alpha(a_1)\,b_1+\cdots+\alpha(a_n)\,b_n$.
  The corresponding unitary group is denoted by $\SaU{n}{\HH,f}$, its
  Lie algebra by~$\rType{D}n{\HH}$ (in the notation of Tits~\cite{MR0218489}; %
  Helgason~\cite{MR514561} denotes the corresponding Lie algebras
  by~$\mathfrak{so}^*(2n)$). %
  We also note that the form~$f$ can be replaced by the skew-hermitian
  form $if$, without changing the unitary group (but changing the
  involution back to the standard one). 
\end{ndef}

\removelastskip%
\enlargethispage{13mm}%
\begin{exam}\label{rem:SaU}
  An interesting special case of our construction occurs if $\F=\CC$,
  $\R=\RR$, $\ell=2$, and $h$ is a hermitian form of Witt index~$1$
  on~$\CC^4$. The algebra~$\K$ is then isomorphic to~$\HH$, the
  involution~$\alpha$ has a $3$-dimensional space of fixed points, and
  is thus a conjugate of the involution used in~\ref{def:SaU}. Our
  result~\ref{homoSUtoSaU} thus provides a covering
  $\SU[]4{\CC,1}\to\SaU{3}{\HH,g}$ corresponding to the isomorphism
  from the Lie algebra~$\rType{A}3{\CC,1}$ onto~$\rType{D}3{\HH}$ (in
  the notation of Tits~\cite{MR0218489}; %
  Helgason~\cite{MR514561} denotes the corresponding Lie algebras
  by~$\mathfrak{su}(3,1)$ and~$\mathfrak{so}^*(6)$, respectively).

  The group $\U[]4{\CC,1}$ is a semidirect product of its commutator
  group $\SU[]4{\CC,1}$ with a group isomorphic to $\RR/\ZZ$. %
  With respect to a basis $v_1,v_2,v_3,v_4$ with
  $h(v_1,v_1)= h(v_2,v_2)= h(v_3,v_3)= 1$ and $h(v_4,v_4)=-1$, we
  consider the linear maps $\gamma_s$ defined by $\gamma_s(v_1)=v_1$,
  $\gamma_s(v_2)=v_2$, $\gamma_s(v_3)=v_3$, and $\gamma_s(v_4)=v_4s$,
  with $s\gal{s}=1$. Then $\set{\gamma_s}{s\in\CC,s\gal{s}=1}$ is a
  complement for $\SU[]4{\CC,1}$ in $\U[]4{\CC,1}$. %
  It turns out that $\eta_2(\gamma_s)$ is an $\HH$-semilinear map with
  companion $\varphi_s\colon x+\bj y \mapsto x+\bj sy$
  (see~\ref{similitudeAlgebraAut}), fixing each vector in the
  $\HH$-basis $v_1\wedge v_2$, $v_1\wedge v_3$, $v_2\wedge v_3$. %

  We also remark that the $\CC$-semilinear map~$\lambda$ with
  companion~$\sigma$ and fixing each $v_k$ induces the
  $\HH$-semilinear map $\eta_2(\lambda)$ with companion
  $\varphi\colon x+\bj y\mapsto \gal{x}+\bj\gal{y}$ fixing each one of
  $v_1\wedge v_2$, $v_1\wedge v_3$, $v_2\wedge v_3$.
\end{exam}

\goodbreak%
\section {The split cases}\label{sec:SplitCases}

We will call $\K_\ell$ \emph{split} whenever it contains divisors of
zero. This extends the established terminology for composition
algebras. Recall that $\K_\ell$ is split precisely if $\delta_\ell$ is
a norm: $\delta_\ell=\gal{s}s$ for some~$s\in\F^\myTimes$ %
(and this happens precisely if~$h$ has
discriminant~$(-1)^\ell$).

In that case, we may assume $\delta=1$ without loss of generality.  In
fact, if we replace our isomorphism $b\colon \Ext[n]V\to\F$ by~$sb$
then the Hodge operator $J$ is replaced by
$J\,s^{-1}\colon X\mapsto J(X)s^{-1}$, and we have
$(J\,s^{-1})^2=\id$.  %
So the subalgebra $M_\ell$ of $\End[\R]{\Ext[\ell]V}$ remains the same
(see~\ref{similitudeAlgebraAut}), while the algebra $\K_\ell$ is
replaced by its conjugate $\left(
    \begin{smallmatrix}
      1 & 0 \\
      0 & s
    \end{smallmatrix}
  \right)^{-1} %
  \K_\ell %
  \left(
    \begin{smallmatrix}
      1 & 0 \\
      0 & s
    \end{smallmatrix}
  \right) %
  = \set{\left(
    \begin{smallmatrix}
      y_0 & \gal{y_1} \\
      y_1 & \gal{y_0}
    \end{smallmatrix}
  \right)}{y_0,y_1\in\F}
$.

\begin{ndef}[Convention]
  For the rest of this section, we will always assume that we have
  normalized~$b$ in such a way that $J^2=\id$ whenever $\K_\ell$ is
  split (so $\delta_\ell=1$).  Then $z\coloneqq 1+\bj_\ell \in\K_\ell$
  satisfies $z^2=2z$.  Thus $z$ is nilpotent if $\Char\F=2$, and
  $p\coloneqq\frac12z$ is an idempotent if $\Char\F\ne2$.
\end{ndef}

\begin{lemm}\label{splitCases}
  If\/ $\K_\ell$ is split then we have one of the following:
  \begin{enumerate}
  \item $\sigma\ne\id$: then $\dim_\R\K_\ell=4$ and
    $\K_\ell\cong\R^{2\times2}$ is the split quaternion algebra
    over~$\R$.
    On $\R^{2\times2}$ the standard involution is given by
    $\kappa\left(
      \begin{smallmatrix}
        a & b \\
        c & d
      \end{smallmatrix}\right)
    = \left(
      \begin{smallmatrix}
        d & -b \\
        -c & a
      \end{smallmatrix}\right)$. 
  \item $\sigma=\id$ and $\Char\F\ne2$: then
    $\K_\ell\cong\F\times\F\cong\F[X]/(X^2-1)$ is the two-dimensional
    split composition algebra over~$\F$.
    On $\F\times\F$ the standard involution is $\kappa(a,b)=(b,a)$. 
  \item $\sigma=\id$ and $\Char\F=2$: then
    $\K_\ell\cong\F[X]/(X^2)$ is a local ring, and
    the standard involution is the identity. 
  \end{enumerate}
\end{lemm}
\begin{proof}
  See~\cite[1.8]{MR1763974} for the first two cases. The last case is
  checked directly.
\end{proof}

\begin{lemm}\label{splitCaseSplitModule}
  Let\/ $W\coloneqq \Ext[\ell]{V}$, and assume that $\K_\ell$ is
  split.

  \begin{enumerate}
  \item If\/ $p\in\K_\ell\smallsetminus\{0,1\}$ is an idempotent then
    $\gal{p}=1-p$ and\/~$W$ decomposes as the
    direct sum $\Ext[\ell]{V} = Wp \oplus W\gal{p}$ of\/
    $\SUV{V,h}$-modules~$Wp$ and\/ $W\gal{p}=W(1-p)$.
  \item\label{conjIdempotents} If $\sigma\ne\id$ then
    $\K_\ell\cong\R^{2\times2}$ contains exactly one conjugacy class
    of idempotents apart from~$0$ and\/~$1$; %
      in~$\K_\ell$ that class is represented by $\left(
      \begin{smallmatrix}
        u & \gal{u}\\
        u & \gal{u}
      \end{smallmatrix}\right)$ with $1=\gal{u}+u$. \\%
    In particular, the $\SUV{V,h}$-modules $Wp$ and\/ $W\gal{p}$ are
    isomorphic.
  \item\label{twoIdempotents} If $\sigma=\id$ and $\Char\F\ne2$ then
    $\K_\ell\cong\F\times\F = \R\times\R $ contains precisely two
    idempotents apart from~$0$ and\/~$1$, corresponding to $(1,0)$ and
    $(0,1)$, respectively.  No element of $\K_\ell\smallsetminus\{0\}$ is
    nilpotent, and the two non-trivial idempotents of~$\K_\ell$ are
    $\frac12(1+\bj_\ell)$ and $\frac12(1-\bj_\ell)$.

    If\/ $\ell$ is odd then the involution $\alpha$ coincides with the
    standard involution~$\kappa$ and interchanges the two
    idempotents. If\/ $\ell$ is even then $\alpha=\id$. 
    
    For any $\K_\ell$-basis $B$ for~$W$ we define
    $\psi_B\colon \sum\limits_{b\in B} b\,x_b \mapsto
    \sum\limits_{b\in B}b\,\gal{x_b}$. %
    Then~$\psi_B$ is a $\K_\ell$-semilinear bijection of\/~$W$ (with
    companion automorphism~$\kappa$) which gives an isomorphism of\/
    $\SUV{V,h}$-modules from $Wp$ onto
    $\gal{Wp} \coloneqq \psi_B(Wp) = W\gal{p}$. %
  \item\label{nilpQuotient} %
    If $\sigma=\id$ and $\Char\F=2$ then $\K_\ell \cong \F[X]/(X^2)$
    does not contain any idempotents apart from~$0$ and\/~$1$. %
    The maximal ideal in~$\K_\ell$ is generated by the nilpotent
    element $z = 1+\bj_\ell$. The submodule $Wz$ and the quotient
    $W/Wz$ are isomorphic via $\rho_z\colon w+Wz\mapsto wz$.
  \end{enumerate}
\end{lemm}

\begin{proof}
  Assume that $p\in\K_\ell\smallsetminus\{0,1\}$ is an idempotent. Then
  $p^2=p$ implies $0=p^2-p=p(1-p)$, and $0 = \det_\F p = \gal{p}p$
  follows. Then $(p+\gal{p})^2= p^2+2\gal{p}p+\gal{p}^2 = p+\gal{p}$
  yields that $p+\gal{p}$ is an idempotent in the field~$\F$. Since
  $p+\gal{p}=0$ would imply $p=p^2=-\gal{p}p=0$ we infer $p+\gal{p}=1$
  and $\gal{p}=1-p$. Thus $W=Wp + W\gal{p}$, and the sum is direct
  because $vp=w\gal{p}$ implies $w=(v+w)p\in W p$ and
  $w\gal{p}\in Wp\gal{p}=\{0\}$. The summands are $\SUV{V,h}$-modules
  because $\K_\ell$ centralizes $\SUV{V,h}$.

  Let $p,q\in\K_\ell\smallsetminus\{0,1\}$ be idempotents, and let
  $\varphi\colon\R^{2\times2}\to\K_\ell$ be an isomorphism of
  algebras. Then both $p$ and $q$ are conjugates of~$\varphi\left(
    \begin{smallmatrix}
      1 & 0 \\
      0 & 0 
    \end{smallmatrix}\right)$; in particular, there exists
  $a\in\K_\ell^\myTimes$ with $q=a^{-1}pa$. Now the map $w\mapsto wa$ is
  $\SUV{V,h}$-equivariant and induces a module isomorphism from
  $Wp=Wa^{-1}p$ onto $Wpa=Wa^{-1}pa=Wq$. %
  Thus assertion~\ref{conjIdempotents} is established.

  Now assume $\sigma=\id$ but $\Char\F\ne2$. The idempotents in
  $\F\times\F$ are easy to find and it is clear that there are no
  nilpotent elements apart from~$0$. %
  As $W$ is a \emph{free} $\K_\ell$-module (see~\ref{def:Kmodule}) it
  is clear that~$B$ and~$\psi_B$ exist with the required properties.
  
  Finally, we assume $\sigma=\id$ but $\Char\F=2$. The assertion about
  idempotents in $\F[X]/(X^2)$ follows since
  $a^2+\F X^2 \ni a^2+b^2X^2 = (a+bX)^2 \in (a+bX)+\F X^2$ implies
  $b=0$ and $a^2=a\in\F$.

  For the last assertion, it suffices to show that the kernel
  of~$\rho_z$ is precisely $Wz$. The inclusion
  $Wz\subseteq\ker{\rho_z}$ is clear from~$z^2=0$. In order to see the
  reverse inclusion, pick an orthogonal basis $v_1,\dots,v_{2\ell}$
  for~$V$. Let $W_1$ be the subspace of~$W$ spanned by those
  $\ell$-vectors $v_{t_1}\wedge\dots\wedge v_{t_\ell}$ with
  $1\in\{t_1,\dots,t_\ell\}$ and let $W^1$ denote the span of those
  with $1\notin\{t_1,\dots,t_\ell\}$.  Then $W=W_1\oplus W^1$ and $J$
  (i.e., right multiplication by~$\bj_\ell$) induces an isomorphism
  from $W_1$ onto~$W^1$.  Right multiplication with $z$ induces the
  operator $\id+J$ which has rank at least~$\frac12\dim{W}$ because
  its restriction to~$W_1$ is injective. Thus
  $\dim\ker{\rho_z}\le\frac12\dim{W}\le\dim{Wz}\le\dim\ker{\rho_z}$
  yields $\ker{\rho_z}=Wz$ as required.
\end{proof}

\begin{rema}
  In $\R^{2\times2}$, the nilpotent elements apart from~$0$ form a
  single conjugacy class, represented by~$\left(
    \begin{smallmatrix}
      0 & 1 \\
      0 & 0
    \end{smallmatrix}\right)$. %
  The conjugacy class of idempotents (different from~$1$ and~$0$) is
  represented by~$\left(
    \begin{smallmatrix}
      1 & 0 \\
      0 & 0
    \end{smallmatrix}\right)$. %
  If $\sigma\ne\id$ and~$\K_\ell$ is split, the corresponding classes
  in~$\K_\ell$ are represented by any $\left(
    \begin{smallmatrix}
      v & \gal{v} \\
      v & \gal{v} 
    \end{smallmatrix}\right)$ %
  with $v\in\K_\ell\smallsetminus\{0\}$ such that %
  $v+\gal{v}=0$ (for a nilpotent element) or %
  $v+\gal{v}=1$ (for an idempotent).
\end{rema}

\begin{lemm}\label{orthWp}
  Assume $\sigma=\id$, $\Char\F\ne2$ and that $\K_\ell$ splits so that
  there exist idempotents~$p$ and $\gal{p} = 1-p$ in
  $\K_\ell\cong\F\times\F$.
  \begin{enumerate}
  \item In any case the restriction of\/ $g$ to the subspace $Wp$ is a
    multiple of\/ $\Ext[\ell]{h}$: \\
    we have $g(Xp,Yp) = \Ext[\ell]{h}(Xp,Yp)\,2p$ for all $X,Y\in W$.
  \item If\/ $\ell$ is even then $Wp$ and $W(1-p)$ are orthogonal with
    respect to~$g$,  %
    and the restrictions of\/~$g$ to~$Wp$ and~$W(1-p)$, respectively,
    are not degenerate.
  \item\label{trivialRestriction}%
    If\/ $\ell$ is odd then the restrictions of\/~$g$ and
    of\/~$\Ext[\ell]{h}$ to~$Wp$ are trivial.
  \end{enumerate}
\end{lemm}
\begin{proof}
  According to~\ref{splitCaseSplitModule}.\ref{twoIdempotents}, we may
  (up to an application of the standard involution) assume
  $p=\frac12(1+\bj_\ell)$.  We compute %
  $p\bj_\ell = \frac1{2}(1+\bj_\ell)\bj_\ell =
  \frac1{2}(\bj^{}_\ell+\bj_\ell^2) = \frac{1}{2}(\bj_\ell^{}+1) =
  p$. This yields $g(Xp,Yp) =
  \Ext[\ell]{h}(Xp,Yp)+\Ext[\ell]{h}(Xp,Yp\bj_\ell^{})\bj_\ell^{-1} =
  \Ext[\ell]{h}(Xp,Yp)(1+\bj_\ell^{-1}) = \Ext[\ell]{h}(Xp,Yp)\,2p$,
  as claimed.
  
  If\/ $\ell$ is even then $\alpha=\id$ and we compute $g(Xp,Y(1-p)) =
  p\,g(X,Y)\,(1-p) = g(X,Y)\,({p-p^2}) = 0$ for arbitrary $X,Y\in
  W$. Thus $Wp\perp W(1-p)$, as claimed. %
  Using~\ref{gAnisotropicIfhIs}, it is easy to see that both the
  restrictions of~$g$  to~$Wp$ and to~$W(1-p)$ are not degenerate. %
  
  If $\ell$ is odd then $\alpha$ interchanges $p$ with $1-p$ and
  $g(Xp,Yp) = \alpha(p)g(X,Y)p = g(X,Y)\,(1-p)p = 0$ holds for all
  $X,Y\in W$. Thus the restriction of $g$ to~$Wp$ is trivial, and so
  is the restriction of~$\Ext[\ell]{h}$. 
\end{proof}

\begin{lemm}\label{totSingWz}
    Assume $\sigma=\id$, $\Char\F=2$ and that $\K_\ell$ splits so that
    there exists a nilpotent element $z\in\K_\ell\smallsetminus\{0\}$. 
    Then the restriction of~$g$ to the subspace~$Wz$ is trivial.
\end{lemm}
\begin{proof}
  We note $\alpha=\id$ and then compute $g(Xz,Yz) = g(X,Y)\,z^2 = 0$.
\end{proof}

\begin{ndef}[Reduction of the form]\label{restrictForm}
  Assume that $\K_\ell$ is split (with $\bj_\ell^2=1$), and consider
  $z=1+\bj_\ell$.  The value of $\alpha(z)$ depends on~$\ell$; we find
  $\alpha(z)=z$ if~$\ell$ is even and $\alpha(z)=1-\bj_\ell$ if~$\ell$
  is odd. In the latter case, we thus have
  $\alpha(z)z=0$.

  In order to understand the restriction of the form~$g$ to the
  submodule~$Wz$ we introduce some notation, as follows. %
    If~$\ell$ is odd, $\sigma\ne\id$, and $\Char\F\ne2$ we pick
    $a_\ell\in\F\smallsetminus\{0\}$ such that $\gal{a_\ell}=
    -a_\ell$. %
    In all other cases, we take $a_\ell\coloneqq 1$. %
    Using the $\R$-linear map
  \[
  r_\ell \colon \K_\ell \to \R \colon x+\bj_\ell y \mapsto
  \left(x+(-1)^\ell y +(-1)^\ell\left(\gal{x+(-1)^\ell y}\right)
  \right) a_\ell 
  \]
  we then find
  \[
  \begin{array}{rcl}
    \alpha(z)(x+\bj_\ell\,y)z &=& \left(1+(-1)^\ell \bj_\ell\right)
    \left(x+\bj_\ell\,y\right) \left(1+\bj_\ell\right) \\
    &=&
    \left((x+(-1)^\ell y)+(-1)^\ell\gal{(x+(-1)^\ell y)}\right)
    \left(1+\bj_\ell\right) \\
    &=&
    r_\ell(x+\bj_\ell\,y)\,\, a_\ell^{-1}z \,.
  \end{array}
  \]
  Putting
  \[
  g^o(Xz,Yz)  \coloneqq  r_\ell^{}\bigl(g(X,Y)\bigr)
  \]
  we obtain a $(-1)^\ell$-symmetric $\R$-bilinear form $g^o$ on~$Wz$.
  Thus there are homomorphisms %
  \[
  \eta^o_{2k}\colon \SUV{V,h}\to\OV{Wz,g^o}
  \quad\text{ or }\quad
  \eta^o_{2k+1}\colon \SUV{V,h}\to\SpV{Wz,g^o} \,,
  \]
  respectively, according to the parity of~$\ell\in\{2k,2k+1\}$. 
\end{ndef}

Note that the $\R$-bilinear form~$g^o$ on~$Wz$ is zero if $\sigma=\id$
and either $\ell$ is odd or $\Char{F}=2$. %
This fits well with the fact
(see~\ref{orthWp}.\ref{trivialRestriction} and~\ref{totSingWz},
cp. also~\ref{HodgeSemiSimilitude}.\ref{skewH}) that even the
restriction of the $\K_\ell$-sesquilinear form~$g$ to~$Wz$ is trivial
in those cases.
  
\begin{lemm}\label{kerEtaO}
  If the $\SUV{V,h}$-module $Wz$ has a complement in $W$ (in
  particular, if $\sigma\ne\id$ or $\Char\F\ne2$,
  see~\ref{splitCaseSplitModule}.\ref{conjIdempotents},
  \ref{splitCaseSplitModule}.\ref{twoIdempotents}) then that
  complement is isomorphic to~$Wz$ and
  $\ker{\eta^o_\ell}=\ker{\eta^{}_\ell}$.

  If there is no complementary $\SUV{V,h}$-module to $Wz$ 
  then
  $\sigma=\id$ and $\Char\F=2$ and $W/Wz$ is isomorphic to $Wz$,
  see~\ref{splitCaseSplitModule}.\ref{nilpQuotient}. In this case, the group
  $\ker{\eta^o_\ell}/\ker{\eta^{}_\ell}$ is an elementary abelian
  $2$-group because it is isomorphic to a subgroup of $\Hom{Wz}{Wz}$.
\qed
\end{lemm}

\begin{rems}\label{rems:specialCases}
  The description in~\ref{restrictForm} is somewhat complicated
  because $\ell$ is arbitrary (and may be odd).  If one studies the
  Hodge operator in order to understand exceptional isomorphisms
  between classical groups then the values $\ell\in\{1,2\}$ give
  the most interesting examples.

  For $\ell=2$, we will see in~\ref{hQuaternionNorm},
  \ref{splitSigmaNotId}, and~\ref{splitSigmaNotIdCharTwo} below that
  the split cases lead to interesting and intimate connections with
  norm forms on composition algebras.

  If $\sigma=\id$ and $\Char\F=2$ then the precise structure of
  $\SUV{V,h} = \SOV{V,h}$ and $\ker{\eta^o_\ell}/\ker{\eta^{}_\ell}$
  depends on the defect of the form~$h$. Details will be given
  in~\cite{KramerStroppel-hodge-char2-arXiv}. %
  Note that a complementary module for $Wz$ may also exist in that
  case; for instance, this will happen if $\SOV{V,h}$ is trivial
  (cf.~\cite[2.4]{KramerStroppel-hodge-char2-arXiv}). %
\end{rems}

\begin{ndef}[Examples: Symmetric bilinear forms on $\noexpand\RR^4$]\label{RR4}
  Let $\F=\RR$ and $\sigma=\id$. We consider the symmetric bilinear
  forms of Witt index $0$ and~$2$, respectively, on $V=\RR^4$. In both
  cases, the form has discriminant~$1$. So~$\K_2$ splits, we have
  $\K_2\cong\RR\times\RR$, and~\ref{restrictForm} applies.

  If~$h$ has Witt index~$0$ then $\SOV{\RR^4,h}$ is the compact form
  of type~$\Type{D}{2} = \Type{A}1\times\Type{A}1$; denoted
  by~$\rType{D}{2}{\RR,0}$ in~\cite{MR0218489}. %
  The homomorphisms obtained by the actions on~$Wz$ and on~$W\gal{z}$
  are the projections onto the two direct factors in
  $\SOV{\RR^4,h}/\langle-\id\rangle
  \cong\SOV{\RR^3,h_0}\times\SOV{\RR^3,h_0}$, where $h_0$ is the
  (essentially unique) symmetric bilinear form of Witt index~$0$
  on~$\RR^3$.

  If~$h$ has Witt index~$2$ then the connected component
  $\EOV{\RR^4,h}$ of $\SOV{\RR^4,h}$ is the split real form of
  type~$\Type{D}{2}$; denoted by~$\rType{D}{2}{\RR,2}$
  in~\cite{MR0218489}. %
  (See Section~\ref{sec:nonSplit} below for a general definition of
  the groups $\EOV{V,h}$ and $\EUV{V,h}$.) %
  The homomorphisms now obtained by the actions on~$Wz$ and
  on~$W\gal{z}$ are the projections onto the two direct factors in
  $\SOV{\RR^4,h}/\langle-\id\rangle
  \cong\SOV{\RR^3,h_1}\times\SOV{\RR^3,h_1}$, where $h_1$ is the
  (essentially unique) symmetric bilinear form of Witt index~$1$
  on~$\RR^3$.

  We have $\OV{\RR^4,h} \cong \Cg2^2\ltimes\EOV{\RR^4,h}$, where
  $\Cg2$ is a cyclic group of order~$2$. %
  More explicitly, assume that $v_1,v_2,v_3,v_4$ is an orthogonal
  basis with $h(v_1,v_1)=1=h(v_2,v_2)$ and $h(v_3,v_3)=-1=h(v_4,v_4)$.

  Consider the matrices
  \[
   B \coloneqq \left(
     \begin{smallmatrix}
      -1 & 0 & 0 & 0 \\
       0 & 1 & 0 & 0 \\
       0 & 0 & 1 & 0 \\
       0 & 0 & 0 & 1 \\
     \end{smallmatrix}
   \right)\,,\quad
   C \coloneqq \left(
     \begin{smallmatrix}
       1 & 0 & 0 & 0 \\
       0 & 1 & 0 & 0 \\
       0 & 0 &-1 & 0 \\
       0 & 0 & 0 & 1 \\
     \end{smallmatrix}
   \right)
   \quad\text{ and }\quad
   T \coloneqq \left(
     \begin{smallmatrix}
       0 & 0 & 1 & 0 \\
       0 & 0 & 0 & 1 \\
       1 & 0 & 0 & 0 \\
       0 & 1 & 0 & 0 \\
     \end{smallmatrix}
   \right)\,.
 \]
 The complement for $\EOV{\RR^4,h}$ in $\OV{\RR^4,h}$ is generated by the
 involutions~$\beta$, $\gamma$ described by~$B$ and~$C$, respectively. %
 Both $\eta^o_2(\beta)$ and $\eta^o_2(\gamma)$ interchange $Wz$ with $W\gal{z}$;
 their product $\eta^o_2(\beta\gamma)$ generates a complement for $\EOV{Wz,g^o}
 \cong \PSL[2]\RR$ in $\SOV{Wz,g^o}\cong\PGL[2]\RR$. %
 The matrix~$T$ describes a similitude~$\tau$ with~$r_\tau=-1$,
 and~$\eta^o_2(\gamma)$ generates a complement for $\SOV{Wz,g^o}$ in $\OV{Wz,g^o}$.
\end{ndef}

\section{The smallest case}  

It may come as a surprise that the case $\ell=1$ (and $n=2$) is by no
means trivial, but leads to interesting isomorphisms of groups. First
of all, we note that $\Ext[1]{V}$ is naturally identified with~$V$
itself, see~\cite[II\,\S\,7.1]{MR979982}, and $\Ext[1]h$ coincides
with~$h$.

Choose an orthogonal basis $v_1,v_2$ with respect to~$h$ in~$V$,
abbreviate $c_1\coloneqq h(v_1,v_1)$ and $c_2\coloneqq h(v_2,v_2)$,
and define $b\colon\Ext[2]V\to\F$ by $b(v_1\wedge v_2)=1$.
Then~\ref{computeHodge} specializes to
$J(v_1)= v_2 \,\frac{-h(v_1,v_1)}{b(v_1\wedge v_2)} = -v_2\,c_1$, and
$J(v_2)=v_1\,\frac{h(v_2,v_2)}{b(v_1\wedge v_2)} = v_1\,c_2$. %
Thus $J(v_1x_1+v_2x_2)=v_1c_2\gal{x_2}-v_2c_1\gal{x_1}$; in
coordinates with respect to the basis $v_1,v_2$, we have
\[
\setlength{\arraycolsep}{.1em}
\begin{array}{rccrcl}
  h\colon \F^2\times\F^2\to\F\colon &
  \left(
  \left(
    \begin{matrix}
      x_1 \\ x_2
    \end{matrix}
  \right)
  ,
  \left(
    \begin{matrix}
      y_1 \\ y_2
    \end{matrix}
  \right)
  \right)
  &\mapsto&
  \left(
  \gal{x_1},\gal{x_2}\strut
  \right)
  \left(
  \begin{matrix}
    c_1 & 0 \\
    0 & c_2
  \end{matrix}
  \right)
  \left(
  \begin{matrix}
    y_1 \\ y_2
  \end{matrix}
  \right)
  &=& c_1\,\gal{x_1}\,{y_1}+c_2\,\gal{x_2}\,{y_2}
  \\[2ex]
  \text{and}\quad
  J \colon \F^2\to\F^2\colon &
  \left(
  \begin{matrix}
    x_1 \\ x_2
  \end{matrix}
  \right)
  &\mapsto&
  \left(
  \begin{matrix}
    0 & c_2 \\
    -c_1 & 0
  \end{matrix}
  \right)
  \left(
  \begin{matrix}
    \gal{x_1} \\ \gal{x_2}
  \end{matrix}
  \right)
  &=&
  \left(
  \begin{matrix}
    c_2\gal{x_2} \\ -c_1\gal{x_1}
  \end{matrix}
    \right)
\,.
\end{array}
\]
Identifying $x_1+\bj x_2 \in \K_1$ with
$v_1(x_1+\bj x_2) = v_1x_1-v_2c_1x_2 \in V$ and using $\Ext[1]h=h$ we
compute $g(X,Y)=\gal{X}c_1 Y$, see~\ref{gDiagonal}.  %
Straightforward verification yields

\begin{lemm}
  The centralizer of~$J$ in the ring~$\F^{2\times2}$ of $\F$-linear
  endomorphisms of~$V$ is %
  \(%
  \L \coloneqq \set{\left(
      \begin{smallmatrix}
        a & -\gal{b}c_2 \\
        bc_1 & \gal{a}
      \end{smallmatrix}\right)}{a,b\in\F} 
  \). %
  The $\R$-algebras~$\L$ and~$\K_1$ are isomorphic; the norm on~$\L$
  is given by the determinant.  %
  The group $\SOV{V,h}$, or $\SUV{V,h}$, respectively, coincides with
  $\Ss_\L \coloneqq \L \cap \SL[2]\F$. %
  \qed
\end{lemm}

Note that~$\K_1$ is split precisely if~$-c_1c_2$ is a norm,
see~\ref{squareHodge} and~\ref{def:algebraK}. In that case, we may
(upon scaling of the form and one of the basis vectors) assume $c_1=1$
and $c_2=-1$; then $1+\bj$ is a divisor of zero in~$\K_1$.

\begin{theo}
  Assume $\ell=1$ and $n=2$. 
  \begin{enumerate}
  \item If\/~$\K_1$ is not split and\/ $\sigma=\id$ then~$\L = \K_1$. We
    obtain a quadratic field extension $\K_1|\F$, and\/ $\SOV{V,h}$ is
    the norm one group\/~$\Ss_{\K_1}$ in~$\K_1$. %
    
    Note that $\Ss_{\K_1}=\{1\}$ if $\K_1|\F$ is inseparable.
  \item If\/~$\K_1$ is not split and\/ $\sigma\ne\id$ then $\K_1$ is a
    quaternion field over~$\R$. We obtain $\L\cong\K_1$ and\/
    $\SUV{V,h} = \Ss_\L\cong\Ss_{\K_1}$. %
    
    Note that\/ $\SUV{V,h}$
    and\/~$\Ss_{\K_1}$ do not coincide in the ring of\/ $\R$-linear
    endomorphisms of\/~$V$; the elements of these two groups may be
    interpreted as left and right multiplications, respectively, by
    elements of norm one in~$\K_1$.
  \item If\/~$\K_1$ is split, $\Char\F\ne2$, and\/ $\sigma=\id$ then~$\L
    = \K_1 \cong \F\times\F$. Thus $\SOV{V,h} = \Ss_\L \cong
    \Ss_{\F\times\F} = \set{(x,x^{-1})}{x\in\F^\myTimes} \cong
    \F^\myTimes$.
  \item If\/~$\K_1$ is split, $\Char\F=2$, and\/ $\sigma=\id$ then $\L =
    \K_1\cong\F[X]/(X^2)$. %
    Thus $\SOV{V,h} = \Ss_{\K_1} \cong \Ss_{\F[X]/(X^2)} =
    \set{1+bX+(X^2)}{b\in\F}$ is isomorphic to the \emph{additive}
    group of~$\F$. 
  \item If\/~$\K_1$ is split and\/ $\sigma\ne\id$ then
    $\K_1\cong\R^{2\times2}$. Thus
    $\SUV{V,h} = \Ss_\L \cong \Ss_{\R^{2\times2}} = \SL[2]\R$ in this
    case. %
    \qed
  \end{enumerate}
\end{theo}

\begin{ndef}[Examples: Symmetric bilinear forms on $\noexpand\RR^2$
  and on $\noexpand\CC^2$]\label{ex:O2RR}%
  We consider $\F = \RR$ and ${\sigma=\id}$. %
  Essentially (i.e., up to similarity), there are two non-degenerate
  symmetric bilinear forms on~$\RR^2$: %
  An anisotropic form~$h_+$, and a hyperbolic form~$h_-$.

  The discriminant of~$h_+$ is~$1$, so $\delta_1=-1$ and $\bj=\left(
    \begin{smallmatrix}
      0 & -1 \\
      1 & 0
    \end{smallmatrix}\right)$ %
  and $\K_1 \cong \CC$ (which comes as no surprise). %
  Note that $\L=\K_1$. %
  We obtain %
  $\SOV{\RR^2,h_+} = \Ss_\L \cong \Ss_{\CC} =
  \set{c\in\CC}{c\,\gal{c}=1}$. %
  In particular, this group is homeomorphic to a circle. %
  As~$h_+$ is (positive or negative) definite, every similitude
  of~$h_+$ has positive multiplier. Thus the multiplier is a square,
  and $\GOV{\RR^2,h_+} = \RR^\myTimes\OV{\RR^2,h_+}$. Via conjugation,
  the scaling factors act trivially on~$M_\ell\cong\K_1\cong\CC$
  (see~\ref{similitudeAlgebraAut}). %
  Each $\gamma\in \OV{\RR^2,h_+}\smallsetminus\SOV{\RR^2,h_+}$ has
  determinant~$1$, and induces complex conjugation on~$\CC$.

  The discriminant of~$h_-$ is $-1$, we obtain $\delta_1=1$,
  $\bj=\left(
    \begin{smallmatrix}
      0 & 1 \\
      1 & 0
    \end{smallmatrix}\right)$ %
  and $\K_1 = \set{\left(
      \begin{smallmatrix}
        a & b \\
        b & a
      \end{smallmatrix}\right)}{a,b\in\RR} %
  \cong \RR\times\RR$ in this case; an explicit isomorphism from
  $\RR\times\RR$ onto~$\K_1$ is given by $(x,y)\mapsto %
  \frac12\left(
    \begin{smallmatrix}
      x & x \\
      x & x
    \end{smallmatrix}\right) %
  + %
  \frac12\left(
    \begin{smallmatrix}
      y & -y \\
      -y & y
    \end{smallmatrix}\right)$. %
  Note that $\L=\K_1$, again. %
  The isomorphism just stated maps $\SOV{\RR^2,h_-} = \Ss_\L$ onto
  $\set{(a,a^{-1})}{a\in\RR^\myTimes} \subset \RR\times\RR$.

  We may assume that the Gram matrix (with respect to the standard
  basis) is $H=\left(
    \begin{smallmatrix}
      0 & 1 \\
      1 & 0
    \end{smallmatrix}\right)$. %
  Then $v_1 \coloneqq \binom11$, $v_2\coloneqq \binom1{-1}$ is an
  orthogonal basis; we have $h(v_1,v_1)=2$ and $h(v_2,v_2)=-2$. Taking
  $b\colon\Ext[2]{\RR^2}\to\RR$ with $b(v_1\wedge v_2)=-2$, we obtain
  $J(v_1)=v_2$ and $J(v_2)=v_1$.  Thus the matrix $\left(
    \begin{smallmatrix}
      1 & 0 \\
      0 & -1
    \end{smallmatrix}\right)$ %
  describes~$J$ with respect to the standard basis.

  In standard coordinates, we now obtain
  $\SOV{\RR^2,h_-} = \set{\left(
      \begin{smallmatrix}
        a & 0 \\
        0 & a^{-1}
      \end{smallmatrix}\right)}{a\in\RR^\myTimes}$. %
  The orthogonal group is the union of $\SOV{\RR^2,h_-}$ and the coset
  $\left(
    \begin{smallmatrix}
      0 & 1 \\
      1 & 0
    \end{smallmatrix}\right) %
  \SOV{\RR^2,h_-}$. %
  Note that conjugation by $\left(
    \begin{smallmatrix}
      0 & 1 \\
      1 & 0
    \end{smallmatrix}\right)$ maps~$J$ to~$-J$.
  (Using the isomorphism above, this corresponds to swapping $(1,-1)$
  with $(-1,1)$ in $\RR\times\RR$.)  %
  For each $a\in\RR^\myTimes$, the matrix $S_a \coloneqq \left(
    \begin{smallmatrix}
      1 & 0 \\
      0 & a
    \end{smallmatrix}\right)$ %
  describes a similitude with multiplier~$a$.  These similitudes
  centralize~$J$.

  On $\CC^2$, there is essentially only one non-degenerate symmetric
  bilinear form. That form is hyperbolic, the arguments for the
  form~$h_-$ above remain valid \emph{verbatim}. %
  Actually, that reasoning holds for every hyperbolic form on a
  two-dimensional vector space over any field~$\F$ with $\Char\F\ne2$.
\end{ndef}

\begin{ndef}[Examples: Hermitian forms on $\noexpand\CC^2$]\label{ex:U2CC}%
  We consider $\F = \CC$, and take complex conjugation for~$\sigma$. %
  Essentially, there are two non-degenerate $\sigma$-hermitian forms
  on~$\CC^2$: an anisotropic form~$h_+$ and a hyperbolic form~$h_-$.

  The discriminant of~$h_+$ is~$1$, so $\delta_1=-1$ and $\bj=\left(
    \begin{smallmatrix}
      0 & -1 \\
      1 & 0
    \end{smallmatrix}\right)$. The algebra~$\K_1$ is not split, so
  $\K_1\cong\HH$, the skew field of Hamilton's quaternions. %
  We obtain $\SUV{\CC^2,h_+} = \Ss_\L \cong \Ss_\HH$. %
  In particular, this group is homeomorphic to a $3$-sphere. %
  We may assume that the Gram matrix of~$h_+$ is $\left(
    \begin{smallmatrix}
      1 & 0 \\
      0 & 1
    \end{smallmatrix}\right)$. %
  Then $J\binom xy=\binom{-\gal{y}}{\gal{x}}$, and the group
  $\UV{\CC^2,h_+}$ is the semi-direct product of $\SUV{\CC^2,h_+}$ and
  $\set{\left(
      \begin{smallmatrix}
        1 & 0 \\
        0 & u
      \end{smallmatrix}\right)}{u\in\CC, u\gal{u}=1}$. %
  In~$M_1$, conjugation by $\left(
    \begin{smallmatrix}
      1 & 0 \\
      0 & u
    \end{smallmatrix}\right)$ maps~$J$ to %
  $u\,\id\circ J\colon \binom xy\mapsto \binom{-u\gal{y}}{u\gal{x}}$;
  see~\ref{semiSimilitudesAct}. %
  As $h_+$ is a (positive or negative) definite form, each similitude
  has positive multiplier. Therefore, each multiplier is a norm, and
  $\gUV{\CC^2,h_+}=\CC^\myTimes\SUV{\CC^2,h_+}$.  Conjugation by
  $c\,\id$ maps $J$ to $c/\gal{c}\,J$, and induces on~$M_1\cong\HH$ an
  inner automorphism (namely, conjugation by~$c$) that fixes
  $\CC\,\id$ pointwise and multiplies each element of
  $\CC\,\id\circ J$ by the factor $c/\gal{c} \in \Ss_\CC$. %
  (Note that, by Hilbert's Theorem 90 (see~\cite[4.31]{MR780184}),
  each element of $\Ss_\CC$ is obtained in this way, as a quotient
  $c/\gal{c}$ for some $c\in\CC^\myTimes$.) %
  
  The discriminant of~$h_-$ is~$-1$, so $\delta_1=1$ and $\bj=\left(
    \begin{smallmatrix}
      0 & 1 \\
      1 & 0
    \end{smallmatrix}\right)$ in this case.
  The algebra $\K_1 = \set{\left(
      \begin{smallmatrix}
        x & \gal{y} \\
        y & \gal{x}
      \end{smallmatrix}\right)}{x,y\in\CC}$ is then isomorphic to
  $\RR^{2\times2}$; in fact, an explicit isomorphism is obtained by
  $\RR$-linear extension of $\left(
    \begin{smallmatrix}
      1 & 0 \\
      0 & 1
    \end{smallmatrix}\right) \mapsto
  \left(
    \begin{smallmatrix}
      1 & 0 \\
      0 & 1
    \end{smallmatrix}\right)$, %
  $\left(
    \begin{smallmatrix}
      i & 0 \\
      0 & -i
    \end{smallmatrix}\right) \mapsto
  \left(
    \begin{smallmatrix}
      0 & -1 \\
      1 & 0
    \end{smallmatrix}\right)$,
  $\bj = \left(
    \begin{smallmatrix}
      0 & 1 \\
      1 & 0
    \end{smallmatrix}\right) \mapsto
  \left(
    \begin{smallmatrix}
      0 & 1 \\
      1 & 0
    \end{smallmatrix}\right)$,
  and %
  $\left(
    \begin{smallmatrix}
      0 &-i \\
      i & 0
    \end{smallmatrix}\right) \mapsto
  \left(
    \begin{smallmatrix}
      1 & 0 \\
      0 &-1
    \end{smallmatrix}\right)$. %
  We find that $\SUV{\CC^2,h_-}$ is isomorphic to
  $\Ss_{\K_1} \cong \Ss_{\RR^{2\times2}} = \SL[2]\RR$. %
  We may assume that the Gram matrix is $\left(
    \begin{smallmatrix}
      1 & 0\\
      0 & -1
    \end{smallmatrix}\right)$. Then $\UV{\CC^2,h_-}$ is the
  semi-direct product of $\SUV{\CC^2,h_-}$ with $\set{\left(
      \begin{smallmatrix}
        1 & 0 \\
        0 & u
      \end{smallmatrix}\right)}{u\in\CC, u\gal{u}=1}$. %
  As above (in the anisotropic case), conjugation by $\left(
    \begin{smallmatrix}
      1 & 0 \\
      0 & u
    \end{smallmatrix}\right)$ maps~$J$ to %
  $u\,\id\circ J\colon \binom xy\mapsto \binom{-u\gal{y}}{u\gal{x}}$. %

  Translated into the algebra $\RR^{2\times2}$ (via the explicit
  isomorphism above) the isomorphism for $u=\cos(t)+i\sin(t)$ is
  obtained as conjugation by $\left(
    \begin{smallmatrix}
      \cos(t/2) & -\sin(t/2) \\
      \sin(t/2) & \cos(t/2)
    \end{smallmatrix}\right)$. %

  For each similitude $\gamma\in\GUV{\CC^2,h_-}$, the multiplier
  $\mu_\gamma$ lies in~$\RR$ because
  $\set{h(v,v)}{v\in\CC^2} \subseteq \RR$. So every possible
  multiplier is obtained from some
  $\gamma\in S \coloneqq \CC^\myTimes\id\cup\CC^\myTimes\left(
    \begin{smallmatrix}
      0 & 1 \\
      1 & 0
    \end{smallmatrix}\right)$, and $\gUV{\CC^2,h_-}$ is the
  semidirect product of $\UV{\CC^2,h_-}$ with~$S$.  The operator
  $J\in M_1$ is centralized by $\left(
    \begin{smallmatrix}
      0 & 1 \\
      1 & 0
    \end{smallmatrix}\right)$, and
  conjugation of~$J$ by $c\,\id$ yields $(c/\gal{c}\,\id)\circ J$.
\end{ndef}

\section{Norm forms of composition algebras}
\label{sec:rangeSplit}

We concentrate on the case $\ell=2$ and $n=2\ell=4$ from now on.  Note
that the factors $(-1)^\ell$ become irrelevant. %
We normalize $b(v_1\wedge v_2\wedge v_3\wedge v_4) = 1$ and then
simplify notation, writing
\[
  \K \coloneqq \K_2, \quad
  \delta \coloneqq \delta_2 = \det(H),
  \quad
  \bj \coloneqq \bj_2
  \quad\text{and}\quad \eta \coloneqq \eta_2;
\]
here $H$ is the Gram matrix describing $h$ with respect to the
orthogonal basis $v_1,v_2,v_3,v_4$, and $\eta=\eta_2$ is the
homomorphism constructed in Section~\ref{sec:modulesQuat}.

In this section we study the split case; our aim is to understand
$\eta^o(\SUV{V,h})$ where~$\eta^o$ is as in~\ref{restrictForm}. %
Recall that $\K$ splits precisely if $\delta = \det(H)$ is a norm,
i.e., if the form $h$ has discriminant
$\disc{h}=1\in\F^\myTimes/N_{\F|\R}(\F^\myTimes)$. %
This excludes one of the three possible values for the Witt index; in
fact, Witt index~$1$ is impossible in the split case:

\begin{lemm}\label{wittIndexSplit}
  If $\disc{h}=1$ and\/~$h$ has positive Witt index then $h$ has
  Witt index~$2$.
\end{lemm}
\begin{proof}
  The Witt index is at most $2$ because the form is not degenerate.

  Assume first that $\Char\F\ne2$ or that $\sigma\ne\id$; then~$h$ is
  trace-valued.  If the Witt index is positive then there exists a
  hyperbolic pair, i.e.  $w_1,w_2\in V$ with $h(w_1,w_1)=0=h(w_2,w_2)$
  and $h(w_1,w_2)=1$. We pick an orthogonal basis $w_3,w_4$ for
  $\{w_1,w_2\}^\perp$ and obtain
  $-h(w_3,w_3)h(w_4,w_4)\in\disc{h}=N_{\F|\R}(\F^\myTimes)$. Thus there
  exists $c\in\F$ with $-h(w_3,w_3)h(w_4,w_4)=\gal{c}c$, the vector
  $w_3\,c+w_4\,h(w_3,w_3)\in\{w_1,w_2\}^\perp$ is isotropic, and the
  Witt index is at least~$2$.

  There remains the case $\Char\F=2$ and $\sigma=\id$. We assume that
  there exists an isotropic vector $w_1\in V\smallsetminus\{0\}$. Then
  there exists $w_2$ with $h(w_1,w_2)=1$ but it will in general not be
  possible to achieve $h(w_2,w_2)=0$. Now $\{w_1,w_2\}^\perp$ is a
  vector space complement to $w_1\F + w_2\F$.  As the Witt index is at
  most~$2$, that complement contains~$w_3$ with $h(w_3,w_3)\ne0$. Thus
  we find an orthogonal basis $w_3,w_4$ for $\{w_1,w_2\}^\perp$ and an
  isotropic vector in $\{w_1,w_2\}^\perp$ as before.
\end{proof}

\begin{rems}\label{wittIndexTwoIsSplit}
  As a partial converse to~\ref{wittIndexSplit} we note that any two
  non-degenerate $\sigma$-hermitian forms of Witt index~$2$ on~$\F^4$
  are isometric; the discriminant is~$1$, and the corresponding Hodge
  operator yields a split algebra~$\K$.

  Note that over any commutative field~$\F$ with $\Char{\F}\ne2$, the
  norm forms of quaternion algebras are characterized (up to
  similitude) among the quadratic forms in four variables by the fact
  that they have discriminant~$1$; see~\cite[3.1.2]{MR3761133} for an
  explicit proof of this known result.  So~\ref{hQuaternionNorm} below
  applies if the role of~$h$ is played by the polarization of such a
  norm form.
\end{rems}

\begin{nthm}[A quaternion algebra]\label{hQuaternionNorm}
  Assume $\Char{\F}\ne2$ and\/ $\sigma=\id$.  If\/ $\K$ splits
  then~$h$ is similar to the polarization\/~$f_N$ of the norm $N$ of
  some quaternion algebra~$\H$ over~$\F$.

  Moreover, the form~$g^o$ on $Wp$ is similar to the
  restriction of\/~$f_N$ to the space $\Pu{\H} \coloneqq
  \smallset{x\in\H\smallsetminus\F}{x^2\in\F}\cup\{0\}$ of pure
  quaternions.

  The group of similitudes of\/~$f_N$ is well understood, it is
  generated by the standard involution together with left and right
  multiplications by invertible elements of the quaternion
  algebra\footnote{ \ %
  See~\cite[V\,(4.2.4), p.\,266]{MR1096299} for the general result,
  or~\cite[5.2]{MR3871471} for quaternion fields,
  or~\cite[4.5.17]{MR4279905} for arbitrary quaternion algebras with
  $\Char\F\ne2$}. %
  The representation on $W = \Ext[2]{V}$ is the sum of two
  representations that are quasi-equivalent to that on the space of
  pure quaternions, cf.~\cite[Sect.\,6]{MR2926161}. %
  In particular, we have
  $\eta^o_2(\GOV{V,h})= \F^\myTimes\,\SOV{\Pu{\H},f_N|_{\Pu{\H}}}$ and
  $\eta^o_2(\SOV{V,h}) = \SOV{\Pu{\H},f_N|_{\Pu{\H}}}$. %
\end{nthm}
\begin{proof}
  The case where $h$ has Witt index~$2$ is covered by the observation
  that the norm form of the split quaternion algebra over~$\F$ is the
  unique non-degenerate quadratic form of Witt index~$2$ in~$4$
  variables. %
  If $h$ is anisotropic then a suitable scalar multiple of $h$ is the
  polarization of the norm of a quaternion field because $\disc{h}=1$;
  see~\cite[3.1.2]{MR3761133}. %

  In any case, we may (upon scaling) assume that $2\,h$ is the
  polarization of the norm~$N$ of~$\H$. There exists an orthogonal
  basis $v_1\coloneqq 1,v_2,v_3,v_4$ such that $v_4=v_2v_3$ and that
  $v_2,v_3,v_4$ span $\Pu{\H}$. Thus
  $h(v_4,v_4) = N(v_4) = N(v_2)\,N(v_3) = h(v_2,v_2)\,h(v_3,v_3)$. Now
  the sequence $X_1 \coloneqq (v_1\wedge v_2)z$,
  $X_2 \coloneqq (v_1\wedge v_3)z$, $X_3 \coloneqq (v_1\wedge v_4)z$
  forms an orthogonal basis for~$g^o$ on~$Wz$ such that
  $g^o(X_1,X_1) = 2\,h(v_2,v_2)$, $g^o(X_2,X_2) = 2\,h(v_3,v_3)$, and
  $g^o(X_3,X_3) = 2\,h(v_2,v_2)\,h(v_3,v_3)$.  Thus $g^o$ is isometric
  to the restriction of~$2\,h$ to~$\Pu{\H}$.
\end{proof}

\begin{rema}
  We have excluded the characteristic~$2$ case
  in~\ref{hQuaternionNorm}. Indeed, the situation becomes more
  complicated, see~\cite{KramerStroppel-hodge-char2-arXiv}.
\end{rema}

\removelastskip%
\enlargethispage{13mm}%
\begin{ndef}[An octonion algebra]\label{octonion}
  We consider $\sigma\ne\id$ and $\Char\F\ne2$ and assume that~$\K$
  splits; so $\K \cong \R^{2\times2}$. %
  Replacing $h$ by a suitable scalar multiple, we may assume that
  there is an orthogonal basis $v_1,v_2,v_3,v_4$ such that
  $h(v_1,v_1)=1$ and $h(v_4,v_4)=h(v_2,v_2)\,h(v_3,v_3)$; this last
  condition can be satisfied because~$h$ has discriminant~$1$. Then we
  may regard the $\R$-linear span
  $\H:= v_1\R\oplus v_2\R\oplus v_3\R\oplus v_4\R$ as a quaternion
  algebra %
  with norm $N_\H(x)=h(x,x)$. Choosing any $q\in\F^\myTimes$ with $\sigma(q)=-q$
  we obtain $V=\H\oplus\H q$ and that the quadratic form
  $N(v):=h(v,v)$ is the norm of the octonion algebra~$\C$ obtained as
  the $q^2$-double of~$\H$; cf.~\cite[1.5.3]{MR1763974}
  or~\cite[p.\,444\,f]{MR780184}.  Finally, we identify $\F$ with the
  subalgebra $v_1\F = v_1\R\oplus v_1q\R$ of~$\C$ and note %
  $(v_1\F)^\perp= v_2\F\oplus v_3\F\oplus v_4\F$.
\end{ndef}

\begin{exas}
  The quaternion algebra~$\H$ in~\ref{octonion} need not be isomorphic
  to the split quaternion algebra~$\K$. %
  In general, the norm form of \emph{every} quaternion field~$B$
  over~$\R$ has discriminant~$1$, and extending that form to a
  hermitian form on the tensor product $V \coloneqq \F\otimes_\R B$ we
  obtain a case where $\H \cong B$ is non-split, and not isomorphic
  to~$\K$.

  For instance, consider $\F = \CC$ and $\R = \RR$. The (essentially
  unique) positive definite hermitian form on~$\CC^4$ has
  discriminant~$1$; then the restriction of the form~$h$ to~$\H$ is
  positive definite, and the quaternion algebra~$\H$ is isomorphic to
  the quaternion field~$\HH$ (and is, in particular, not split in this
  case).  The hermitian form of Witt index~$2$ also has
  discriminant~$1$, but the quaternion algebra~$\H$ is split in that
  case (and then isomorphic to~$\RR^{2\times2} \cong \K$).
\end{exas}

\begin{theo}\label{splitSigmaNotId}
  Assume $\sigma\ne\id$, $\Char\F\ne2$ and that\/~$\K$
  splits. Let\/~$f_N$ denote the polarization of the norm of the
  octonion algebra~$\C$ introduced in~\ref{octonion}. %
  Then the form~$g^o$ on $Wz=Wp$ is equivalent to the %
  restriction of\/~$f_N$ to~${\F^\perp}$ in that octonion algebra. %
  So $\eta^o_2$ may be interpreted as a homomorphism %
  $\eta^o_2\colon \SUV{V,h} \to \OV{\F^\perp,N|_{\F^\perp}}$.
\end{theo}
\begin{proof}
  Let $c_k:=h(v_k,v_k)$. We normalize $b$ such that $b(v_1\wedge
  v_2\wedge v_3\wedge v_4)=1$. %
  Straightforward computation yields the values of $g^o$ on the basis
  \[
  \begin{array}{lll}
  X_1:=(v_1\wedge v_2)z    , &
  X_2:=(v_1\wedge v_3)z    , &
  X_3:=(v_1\wedge v_4)z    , 
  \\ 
  X_4:=(v_1q \wedge v_2)z, &
  X_5:=(v_1q \wedge v_3)z, &
  X_6:=(v_1q \wedge v_4)z;
\end{array}
\]
  it turns out that this is an orthogonal basis with
\[
\setlength{\arraycolsep}{.1em}
\begin{array}[b]{rcrcrcrcrcr}
  g^o(X_1,X_1) & = & 2 c_2,  & \phantom{xx}&
  g^o(X_2,X_2) & = & 2 c_3,  & \phantom{xx}&
  g^o(X_3,X_3) & = & 2 c_4, \\
  g^o(X_4,X_4) & = & -2q^2c_2,  & \phantom{xx}&
  g^o(X_5,X_5) & = & -2q^2c_3,  & \phantom{xx}&
  g^o(X_6,X_6) & = & -2q^2c_4. 
\end{array}
\]
So~$g^o$ is equivalent to the polarization of $N|_{\F^\perp}$.
\end{proof}

\begin{rema}\label{rem:notSurjective}
  Depending on the form~$h$ (and not only on the ground field~$\F$),
  the image $\eta_2^o(\SUV{V,h})$ may coincide with
  $\SOV{Wp,g^o} \cong \SOV{\F^\perp,f_{N|_{\F^\perp}}}$, or may form a
  \emph{proper} subgroup. See~\ref{SU4ContoSO6R},
  \ref{SU4C2ontoSO6R2}, and~\ref{SU4finiteontoOminus6} below for
  examples.

  As $\SOV{Wp,g^o}$ is generated by its half-turns
  (cf.~\ref{generateSO} below) this means that in some cases elements
  of $\SUV{V,h}$ will induce all these half-turns, while in some other
  cases they will not.
\end{rema}

\begin{ndef}[Example: Compact groups]\label{SU4ContoSO6R}%
  On $\CC^4$ the standard (positive definite) hermitian form~$h_0$
  defines the compact real form $\SU[]{4}{\CC}$ of the simply
  connected complex Lie group of type~$\Type{A}{3}$. %
  The discriminant of~$h_0$ is~$1$, the quadratic form $h_0(v,v)$ is
  the norm of \emph{the} quaternion field~$\HH$ (over the reals),
  and~$\K$ splits. %
  The octonion algebra in~\ref{octonion} is the octonion field~$\OO$,
  and the form $N|_{\CC^\perp}$ is a positive definite quadratic form
  on~$\RR^6$ defining the compact form $\SO{6}{\RR}$ of the complex
  Lie group of type~$\Type{D}{3}$. Of course, types~$\Type{A}{3}$ and
  $\Type{D}{3}$ are identical because the Coxeter diagrams are the
  same. The homomorphism~$\eta^o_2$ induces a two-sheeted covering
  from $\SU[]{4}{\CC}$ onto $\SO{6}{\RR}$.
\end{ndef}

\goodbreak%
\begin{ndef}[Example: Non-compact groups]\label{SU4C2ontoSO6R2}
  Let $h$ be a hermitian form of Witt index~$2$ on~$\CC^4$. Then
  $\disc{h}=1$ and $\K$ is split. The quadratic form $N(v):=h(v,v)$ is
  the norm form of the split quaternion algebra $\RR^{2\times2}$, and
  the octonion algebra constructed in~\ref{octonion} is split, as
  well. The norm form of that octonion algebra has the (maximal
  possible) Witt index~$4$ but its restriction to $\CC^\perp$ only has
  index~$2$. In fact, a $3$-dimensional positive definite subspace in
  $\CC^\perp$ would sum up with $\CC$ to form a $5$-dimensional
  positive definite subspace in the split octonion algebra, which is
  impossible. Thus~$\eta^o_2$ yields a homomorphism from
  $\SU[]{4}{\CC,2}$ to $\SO{6}{\RR,2}$. %
  This homomorphism is not surjective; %
  in fact~$\SU[]{4}{\CC,2}$ and its image are connected but
  $\SO{6}{\RR,2}$ is disconnected (see~\cite{MR101269}
  or~\cite{MR1308913}; in fact, %
  the commutator group of $\SO{6}{\RR,2}$ coincides with the kernel of
  the spinor norm, and has index~$2$ in $\SO{6}{\RR,2}$ because there
  are two square classes; see~\cite[\S8, p.\,54]{MR0310083},
  \cite[Th.\,9.7, p.\,77]{MR1859189}). %
  The tangent object of that morphism of Lie groups is the isomorphism
  between the real forms of types~$\rType{A}{3}{\RR,2}$
  and~$\rType{D}{3}{\RR,2}$, respectively. %
\end{ndef}

\begin{ndef}[An octonion algebra in characteristic two]%
  \label{octonionCharTwo}
  We consider $\sigma\ne\id$ again, but assume $\Char\F=2$ now. %
  The quadratic form $N(v)\coloneqq h(v,v)$ is not degenerate; in
  fact, its polarization
  $f_N(v,w)=h(v,w)+h(w,v) = h(v,w)+\gal{h(v,w)}$ is not degenerate
  (cf.~\cite[4.3]{MR2942723}).

  Pick $u\in F$ with $u+\gal{u}=1$, choose an orthogonal basis
  $v_1,\dots,v_4$ in~$V$, and put $c_k\coloneqq N(v_k)$ for $1\le
  k\le4$. Passing to a suitable scalar multiple of~$h$, we may assume
  $c_1=1$. As~$h$ has discriminant~$1$, we may then further assume
  $c_4=c_2c_3$.

  Let~$\C_1$ be the $\R$-linear span of~$v_1$ and $v_1u+v_2$. Then the
  restriction $N|_{\C_1}$ is equivalent to the norm form of the
  two-dimensional composition algebra $\R[X]/(X^2-X+r)$ with
  $r\coloneqq\gal{u}u+c_2$. %
  The spaces %
  $\C_2\coloneqq v_2 \R + (v_1 c_2 + v_2 u)\R$, %
  $\C_3\coloneqq v_3 \R + (v_3 u + v_4 )\R$, and %
  $\C_4\coloneqq v_4 \R + (v_3 c_2 + v_4 u)\R$ %
  yield an orthogonal decomposition
  $V=\C_1\operp\C_2\operp\C_3\operp\C_4$. %
  For each $k\le4$, the restriction $N|_{\C_k}$ is equivalent to
  $c_kN|_{\C_1}$. In particular, the restriction $N|_{\C_1+\C_2}$ is
  equivalent to the norm of the quaternion algebra~$\H$ obtained as
  the $c_2$-double of $\R[X]/(X^2-X+r)$, and~$N$ itself is equivalent
  to the norm of the octonion algebra~$\C$ obtained as the
  $c_3$-double of~$\H$.
\end{ndef}

\begin{rema}
  Clearly, the restriction $N|_{v_1\F}$ of the norm in~$\C$ to the
  subspace~$v_1\F$ is isometric to the norm form $N_{\F|\R}$. Indeed,
  one may choose the multiplication (of~$\C$) on~$V$ in such a way
  that~$v_1\F$ forms a subalgebra isomorphic to~$\F$. %
  Note that the restriction $N|_{\C_1}$ need not be isometric to the
  norm form $N_{\F|\R}$. In fact, the form $N|_{\C_1}$ may be
  isotropic; for instance, this happens if $c_2=\gal{u}u$.
\end{rema}

\removelastskip%
\enlargethispage{5mm}%
We now use the quadratic form~$\pq$ that gives rise to the Klein
quadric, see~\ref{KleinQuadric}: %
Up to a scalar, we have $\pq(X)^2=\det{X}$ for each
$X\in\Ext[2]{\F^4}$, and this determines~$\pq$ (again, up to a scalar)
because square roots are unique (if existent) in characteristic~$2$.

\removelastskip%
\enlargethispage{5mm}%
\begin{theo}\label{splitSigmaNotIdCharTwo}
  Assume $\sigma\ne\id$, $\Char\F=2$ and that\/~$\K$ splits. %
  Let\/~$f_N$ denote the polarization of the norm~$N$ of the octonion
  algebra~$\C$ introduced in~\ref{octonionCharTwo}. %
  Then the restriction of the quadratic form~$\pq$ to~$Wz$ is similar
  to the restriction of the norm~$N$ to~${\F^\perp}$ in that octonion
  algebra.  In particular, the form~$g^o$ on~$Wz$ is a non-trivial
  scalar multiple of the restriction of\/~$f_N$ to~${\F^\perp}$. %

  Thus $\eta^o_2$ may be interpreted as a homomorphism
  $\eta^o_2\colon \SUV{V,h} \to \OV{\F^\perp,N|_{\F^\perp}}$. %
\end{theo}
\begin{proof}
  Notation and normalizations are as in~\ref{octonionCharTwo}. %
  For~$Wz$, we use the $\R$-basis
  \[
  ({v_1\wedge v_2})z, %
  \quad%
  ({v_1\wedge v_3})z, %
  \quad%
  ({v_1\wedge v_4})z, 
  \quad%
  ({v_1u\wedge v_2})z, %
  \quad%
  ({v_1u\wedge v_3})z, %
  \quad%
  ({v_1u\wedge v_4})z %
  , %
  \] %
  and claim that $\R$-linear extension of %
  $\psi((v_1\wedge v_k)z) \coloneqq v_k$ and %
  $\psi((v_1u\wedge v_k)z) \coloneqq v_ku$ %
  (for $k\in\{2,3,4\}$, respectively) gives an isometry
  from~$\pq|_{Wz}$ to~$N|_{\F^\perp}$.

  The values of~$\pq$ are %
  $\pq((v_1 \wedge v_2)z) = c_2 $, %
  $\pq((v_1 \wedge v_3)z) = c_3 $, %
  $\pq((v_1 \wedge v_4)z) = c_2c_3 $, %
  $\pq((v_1u\wedge v_2)z) = \gal{u}uc_2 $, %
  $\pq((v_1u\wedge v_3)z) = \gal{u}uc_3 $, %
  and %
  $\pq((v_1u\wedge v_4)z) = \gal{u}uc_2c_3 $. %
  Computing $f_{\pq}(Xz,Yz) = \pq(Xz+Yz)-\pq(Xz)-\pq(Yz)$ for $Xz,Yz$
  in the basis we see that $\psi$ is an isometry, as claimed.

  Using~\ref{restrictForm} we compute $g^o(Xz,Yz) =
  \Ext[2]{h}(X,Y)+\gal{\Ext[2]{h}(X,Y)} = f_{\pq}(Xz,Yz)$ for
  arbitrary members $X,Y$ in the basis.  %
\end{proof}

\begin{ndef}[Example: Finite groups]\label{SU4finiteontoOminus6}
  Let~$\F$ be a finite field of square order~$e^2$. Then
  $\F\cong\FF_{e^2}$, there is a unique involution (namely,
  $\sigma\colon x\mapsto x^e$) in $\Aut{\F}$, %
  and the algebra~$\K_2$ is split (as is every finite quaternion
  algebra).

  The (essentially unique) non-degenerate $\sigma$-hermitian form %
  $h\colon\F^4\times\F^4\to\F$ has Witt index~$2$, and the special
  unitary group $\SUV{4,e^2} \coloneqq \SUV{\F^4,h}$ has order %
  $\left|\SUV{4,e^2}\right| = e^{6}\, (e^2-1)(e^3+1)(e^4-1)$.

  The norm form on the (necessarily split) octonion algebra~$\C$
  over~$\R$ has maximal Witt index. Constructing that algebra by
  repeated doubling, starting with the separable extension~$\F|\R$, we
  find that the norm form on $\F^\perp\le\C$ is a quadratic form of
  Witt index~$2$. The corresponding orthogonal group is usually
  denoted by $\OmV{6,e}$; it has order $|\OmV{6,e}| =
  2\,e^{6}\,(e^{2}-1)(e^{3}+1)(e^{4}-1)$, cf.~\cite[9.11,
  14.48]{MR1859189}.

  As $\SUV{4,e^2}$ is a perfect group (cf.~\cite[11.22]{MR1859189}),
  the image $\eta^o_2(\SUV{4,e^2})$ is contained in the commutator
  group $\OpmV{6,e}$ of $\OmV{6,e}$. The order of $\OpmV{6,e}$ is
  $\frac14|\OmV{6,e}|$ if~$e$ is odd, and it is
  $\frac12|\OmV{6,e}|$ %
  if~$e$ is even, cf.~\cite[6.28, 9.7, 9.11, 14.49]{MR1859189}. %
  In any case, the
  map~$\eta^o_2$ found in~\ref{splitSigmaNotId}
  and~\ref{splitSigmaNotIdCharTwo} is a homomorphism $\eta^o_2\colon
  \SUV{4,e^2} \to \OpmV{6,e}$.

  \begin{enumerate}
  \item%
    If~$e$ is odd then the kernel of~$\eta^o_2$ has order~$2$,
    see~\ref{kernelEta}. %
    So~$\eta^o_2$ yields an isomorphism from
    $\SUV{4,e^2}/\langle-\id\rangle$ onto $\OpmV{6,e}$.
  \item%
    If~$e$ is even then~$\eta^o_2$ is injective by~\ref{kernelEta},
    and~$\eta^o_2$ yields an isomorphism from $\SUV{4,e^2}$ onto
    $\OpmV{6,e}$. %
  \end{enumerate}
  See also Taylor~\cite[p.\,198]{MR1189139} for a proof of
  $\SUV{4,e^2} \cong \OpmV{6,e}$ for even~$e$. %
\end{ndef}

\section{Geometry}

We continue to assume $\ell=2$ and $n=2\ell = 4$. After
Section~\ref{sec:rangeSplit} it remains to consider the case where
$\delta=\delta_2$ is not a norm; so $\K=\K_2$ is a division algebra. %

The Klein quadric
$\smallset{X\F}{X\in\Ext[2]{V}\smallsetminus\{0\},\,\pq(X)=0}$
provides a model for the space of lines in the projective
space~$\PG{V}$ (see~\cite[Sect.\,3]{MR2431124}
or~\cite[Ch.\,12]{MR1189139}); %
one maps $u\F\oplus v\F$ to $({u\wedge v})\F$.  Recall that points on
that quadric represent confluent lines if, and only if, they are
orthogonal with respect to~$\pf$, the polarization of~$\pq$. %
There are two types of maximal totally isotropic subspaces,
corresponding to line pencils (i.e., points) and line spaces of planes
in $\PG{V}$, respectively. Every semi-similitude of~$\pf$ thus induces
either a collineation or a duality of~$\PG{V}$.

\begin{lemm}\label{collUnique}
  Consider a vector space~$U$ of dimension $d\ge2$ over~$\F$, and
  let\/ $f\colon U\times U\to\F$ be a non-degenerate diagonalizable
  $\sigma$-hermitian form.  Two collineations from~$\PG{U}$ onto some
  projective space~$\cP$ are equal if they agree on the
  complement\/~$\cR_U$ of the set\/ %
  $\smallset{v\F}{v\in U,f(v,v)=0}$.
\end{lemm}
\begin{proof}
  Let $v_1,\dots,v_d$ be an orthogonal basis with respect to the
  form~$f$.

  Assume first that~$\F$ has more than two elements.  Each line
  meeting~$\cR_U$ then contains at least~$2$ points of~$\cR_U$. For
  any point $v\F \in \PG{U}\smallsetminus\cR_U$ there are at least two
  lines joining~$v\F$ with points in $\{v_1\F,\dots,v_d\F\}$. As each
  of these lines contains at least two points of~$\cR_U$, the images
  of~$v\F$ under the considered collineations are uniquely determined
  by the images of points in~$\cR_U$.

  It remains to study the case where~$\F=\FF_2$. Then $\sigma=\id$,
  and $f(x,y)=\sum_{k=1}^dx_ky_k$ for $x=\sum_{k=1}^dv_kx_k$ and
  $y=\sum_{k=1}^dv_ky_k$. Therefore, we have
  \[
  x\FF_2 \in\cR_U \iff
  \sum_{k=1}^dx_k^2 \ne0 \iff \sum_{k=1}^dx_k \ne0 \,.
  \]
  This means that
  $\PG{U}\smallsetminus\cR_U$ is a hyperplane in~$\PG{U}$, and the
  assertion of the lemma follows from the fact that projective
  collineations are determined by restrictions to affine spaces
  (obtained by removing hyperplanes).
\end{proof}

\begin{theo}\label{JinducesPerp}
  The Hodge operator induces the polarity $\pi_h\colon
  U\leftrightarrow U^{\perp_h}$ on~$\PG{V}$.
\end{theo}
\begin{proof}
  From~\ref{HodgeSemiSimilitude} we know that $J$ is a semi-similitude
  of~$\pf$. Like any semi-similitude of the Pfaffian form~$\pf$, the
  map $J$ induces either a collineation or a duality of the projective
  space~$\PG{V}$. %
  Let $v_1,v_2,v_3,v_4$ be an orthogonal basis with respect to~$h$. %
  The $\F$-subspaces $M$ and $J(M)$ spanned by
  $\{{v_1\wedge v_2},\, {v_1\wedge v_3},\, {v_1\wedge v_4}\}$ and by
  $\{{v_3\wedge v_4},\, {v_2\wedge v_4},\, {v_2\wedge v_3}\}$,
  respectively, are complementary maximally $\pq$-isotropic subspaces
  of~$\Ext[2]{V}$. One of these corresponds to a point, the other to a
  plane in~$\PG{V}$. (Cf.\ also~\cite[7.2]{MR0470099}.) %
  Since~$J$ interchanges the two it induces a duality (and not a
  collineation), which is a polarity because $J^2\in\F\,\id$. %

  From~\ref{computeHodge} we recall that~$J(v_k\wedge v_m)$
  corresponds to the line $\{v_k,v_m\}^\perp = \pi_h(v_k\F+v_m\F)$.

  We will apply~\ref{collUnique} now, with restrictions of~$h$ to
  various subspaces~$U$ playing the role of~$f$. %
  Assume first that either $\Char\F\ne2$ or~$\sigma\ne\id$. %
  Put $U=v_k^\perp$ for some $k\le4$.  For each point
  $v\F \in \PG{U}\smallsetminus\cR_U$, we then know that it is possible to
  extend the orthogonal system $v_k,v_m$ to an orthogonal basis
  containing~$v$ and see that~$J$ and~$\pi_h$ agree on the line
  corresponding to $(v_k\wedge v)\F$. From~\ref{collUnique} we now
  infer that~$J$ and~$\pi_h$ agree on the points of the projective
  plane~$\PG{U}=\PG{v_k^\perp}$. As this is true for each~$k$, the
  assertion of the theorem follows.

  Now consider the case where $\Char\F=2$ and $\sigma=\id$. %
  Fix $k,m\le4$ and choose $k',m'$ such that
  $\{k,m,k',m'\}=\{1,2,3,4\}$. For any $x = v_kx_k+v_mx_m
  \in\{v_{k'},v_{m'}\}^\perp$ we find that $y\coloneqq
  v_kx_mh(v_m,v_m)-v_mx_kh(v_k,v_k)$ satisfies $h(x,y)=0$ and $h(y,y) =
  h(v_k,v_k)h(v_m,v_m)h(x,x) \in h(x,x)\F^\myTimes$. So
  $v_{k'},x,y,v_{m'}$ forms an orthogonal basis whenever
  $h(x,x)\ne0$. This means that~$J$ maps the line $v_{k'}\F+x\F$ to
  the line $y\F+v_{m'}\F = \pi_h(v_{k'}\F+x\F)$, and~$J$ coincides
  with~$\pi_h$ on $\PG{v_k\F+v_m\F}\smallsetminus\cR_{v_k\F+v_m\F}$.

  The solutions $x\in v_k\F+v_m\F$ for $h(x,x)=0$ form a subspace of
  dimension at most one (here we use $\Char\F=2$ and $\sigma=\id$, and
  also that the restriction of~$h$ to the line is
  diagonalizable). Thus $\PG{v_k\F+v_m\F}\cap\cR_{v_k\F+v_m\F}$
  contains at most one point, and~$J$ coincides with~$\pi_h$ on
  $\PG{v_k\F+v_m\F}$.  Now the two polarities coincide on each point
  row of the six lines of form $v_k\F+v_m\F$, and the assertion of the
  theorem follows also in this case. %
\end{proof}

The \emph{absolute points} of the polarity~$\pi_h$ are the
one-dimensional subspaces of~$V$ that are isotropic with respect
to~$h$. We write~$\Abs{h}$ for the set of absolute points.
Analogously, we write $\Abs{g}$ for the set of absolute points of the
polarity $\pi_g\colon C\mapsto X^{\perp_g}$ of the projective
plane~$\PG[\K]{\Ext[2]V}$.

\begin{prop}\label{lambdaFibres}
  The map %
  $\lambda\colon\Gr[\F]{2}{V}\to\Gr[\K]{1}{\Ext[2]{V}}\colon %
  u\F\oplus v\F \mapsto (u\wedge v)\K$ is well defined, and we have
  the following.
  \begin{enumerate}
  \item If the restriction of~$h$ to $L\in\Gr[\F]{2}{V}$ is
    non-degenerate then $\lambda(L)$ has precisely two
    preimages under~$\lambda$, namely, $L$ and $L^{\perp_h}$. %
    In this case, the Pfaffian~$\pf$ has non-degenerate restriction
    to~$\lambda(L)$, and the $\K$-space $\lambda(L)$ is not isotropic
    with respect to~$g$. 
  \item If the restriction of~$h$ to $L\in\Gr[\F]{2}{V}$ is degenerate
    then $\lambda(L)$ has more than two preimages. %
    In fact, the preimages of\/~$L$ are the elements of the pencil of
    tangents $\Gr[\F]{2}{A^{\perp_h}}$ to the absolute
    point~$A\coloneqq L\cap L^{\perp_h} \in \Abs{h}$; %
    that set is in bijection with the projective line over~$\F$. %
    In this case, the $\K$-space $\lambda(L)$ is isotropic with
    respect to~$g$; i.e., it belongs to the set $\Abs{g}$ of absolute
    points of the polarity induced by~$g$ on the projective space
    over~$\K$. %
  \item For each $Z\in\Ext[2]V$ with $g(Z,Z)=0$ there exist $z,w\in V$
    such that $Z=z\wedge w$ and $h(z,z) = 0 = h(z,w)$. %
    In particular, we have
    $\Abs{g}\subset \lambda\left(\Gr[\F]{2}{V}\right)$.
  \end{enumerate}
\end{prop}
\begin{proof}
  Let $L=u\F\oplus v\F$.  If $L\cap L^{\perp_h}=\{0\}$ then $\pf$ is
  non-degenerate on $\lambda(L)$ because
  $\pf({u\wedge v},{u\wedge v}) = %
  \pf(({u\wedge v})\bj,({u\wedge v})\bj) %
  = 0 \ne \Ext[2]{h}({u\wedge v},{u\wedge v}) %
  = \pf({u\wedge v},({u\wedge v})\bj)$. %
  Thus $\lambda(L)$ meets the Klein quadric in two points; these are
  $(u\wedge v)\F$ and $(u\wedge v)\bj\F$. %
  If $L\cap L^{\perp_h}\ne\{0\}$ we may assume $h(u,v)=0=h(v,v)$.  For
  $X\in(ux_u+v\F)\wedge({uy_u+v\F})$ we obtain %
  $\Ext[2]{h}\left(X,X)\right) = \det\left(
    \begin{smallmatrix}
      x_u^2 & x_uy_u \\
      y_ux_u & y_u^2
    \end{smallmatrix}\right) = 0$. %
  Thus
  $g(\lambda(L)\times\lambda(L))=\{0\}$ follows from~\ref{isotropic}
  since $\pf(u\wedge v,u\wedge v)=0$.

  For the last assertion, we note that $g(Z,Z)=0$ means $0 =
  \pf(Z,Z)$ (whence $Z = u\wedge v$ for some $u,v\in V$) and then $0 =
  \Ext[2]h(Z,Z) = \Ext[2]h(u\wedge v,u\wedge v) = \det\left(
    \begin{smallmatrix}
      h(u,u) & h(u,v) \\
      h(v,u) & h(v,v)
    \end{smallmatrix}
  \right)$. %
  So the restriction of~$h$ to
  $u\F+v\F$ is degenerate, and we find $z,w$ as required.
\end{proof}

\section{The range of our homomorphism in the non-split case}
\label{sec:nonSplit}

We consider the case where $\ell=2$ and $\delta=\delta_2$ is not a
norm. Moreover, we assume that the form~$h$ is isotropic (see
Section~\ref{sec:anisotropic} for the anisotropic case).
From~\ref{wittIndexSplit} and~\ref{wittIndexTwoIsSplit} we know that
these assumptions are equivalent to the assumption  that the Witt index
is precisely~$1$. %
Again, our aim is to understand $\eta(\SUV{V,h})$ where $\eta=\eta_2$
is the homomorphism constructed in Section~\ref{sec:modulesQuat}.

\begin{ndef}[Notation]
  An Eichler (or Siegel) transformation
  (cf.~\cite[p.\,214\,f]{MR1007302}) of $(V,h)$ is given as
  \[
  \Sigma_{z,w,p} \colon V\to V\colon x \mapsto x + z\,h(w,x) -
  (w + z\,p)\,h(z,x)
  \]
  where $z,w\in V$ satisfy $h(z,z) = 0 = h(z,w)$ and
  $\sigma(p)+p = h(w,w)$.  The special case $w=0$ leads to %
  an isotropic transvection~$\Sigma_{z,0,p}$, with $\sigma(p)=-p$.

  Conversely, every transvection in $\UV{V,h}$ is of the
  form~$\Sigma_{z,0,p}$ with $h(z,z)=0$ and $\sigma(p)=-p$ (see, for
  instance, \cite[p.\,94]{MR1859189}). %
  Note that only the trivial isotropic transvection exists if
  $\sigma=\id$ and $\Char{\F}\ne2$. This is the reason why we also
  have to study the more general Eichler transformations.  If
  $\Char\F=2$ and $\sigma=\id$ then there do exist isotropic
  transvections. %

  By $\EOV{V,h}\le\OV{V,h}$ and $\EUV{V,h}\le\UV{V,h}$ we denote the
  subgroups generated by all Eichler transformations. Note that
  $\EUV{V,h}$ is generated by its isotropic transvections
  (see~\cite[6.3.1]{MR1007302}), except if $V=\FF_4^3$. %
  (However, that group is not of interest here because~$3$ is odd.)
\end{ndef}

\begin{lemm}\label{findTransvection}
  If $\sigma\ne\id$ and $\Char\F\ne2$ then the image $\eta(\SUV{V,h})$
  in $\UV{\Ext[2]V,g}$ contains each isotropic transvection.
\end{lemm}
\begin{proof}
  The transvections in $\UV{\Ext[2]V,g}$ are of the form
  \[
  \Sigma_{Z,0,p}\colon\Ext[2]{V}\to\Ext[2]{V}\colon X\mapsto
  X - Zp\,g(Z,X)
  \]
  where $Z\in\Ext[2]{V}$ and $p\in\K$ satisfy $g(Z,Z)=0=\alpha(p)+p$.
  Our assumption $\Char{\F}\ne2$ implies that $\alpha(p)+p=0$ is
  equivalent to $p\in\F$ and $\sigma(p)+p=0$.  %
  There exist $z,w\in V$ with $h(z,z)=0=h(z,w)$ such that $Z=z\wedge
  w$, cf.~\ref{lambdaFibres}.

  For fixed $z\in V$ the alternating bilinear map %
  $(x,y)\mapsto \left(x\,h(z,y)-y\,h(z,x)\right)\wedge z$ has its
  range in~$z^{\perp_h}\wedge z$.  If $h(z,z)=0$ then for each $w\in
  z^\perp\smallsetminus z\F$ the latter set is contained in %
  $(z\wedge w)\K=Z\K$, cf.~\ref{lambdaFibres}. %
  Thus we obtain an $\F$-linear map %
  \[
  \lambda_z\colon\Ext[2]{V}\to Z\K \colon
  x\wedge y \mapsto \left(x\,h(z,y)-y\,h(z,x)\right)\wedge z
  \,.
  \]
  For $q\in\F$ with $\sigma(q)=-q$, the image
  $\eta(\Sigma_{z,0,q}) \in \UV{\Ext[2]V,g}$ is the linear extension
  of %
  $({x\wedge y})\mapsto ({x\wedge y})+\lambda_z({x\wedge y})\,q$.
  This is a transvection because $h(z,z)=0$ yields
  $\lambda_z(z\wedge w)=0$, and it is, of course, in the unitary
  group.  Thus $\eta(\Sigma_{z,0,q})$ is of the form
  $\Sigma_{z\wedge w,0,p'} = \Sigma_{Z,0,p'}$ with
  $p'\in \smallset{z\in\K}{\alpha(z)=-z}$. Since $p\R=p'\R$ we may
  achieve $\eta\left(\Sigma_{z,0,q}\right) = \Sigma_{Z,0,p}$, as
  required.
\end{proof}

\begin{rema}\label{charTwoA3inC3}
  If $\Char{\F}=2$ then $\alpha(p)+p=0$ is equivalent to $p \in
  \R+\bj\F$.  %
\end{rema}

\begin{theo}\label{nonsplitHermitian}
  Assume $\Char\F\ne2$, $\sigma\ne\id$ and that~$\K$ is not split. %
  Then the image of\/ $\SUV{V,h}$ under $\eta$ in $\UV{\Ext[2]V,g}$ is
  the group $\EUV{\Ext[2]V,g}$ generated by all isotropic
  transvections.
\end{theo}
\begin{proof}
  From~\ref{findTransvection} we know that $\eta(\SUV{V,h})$ contains
  $\EUV{\Ext[2]V,g}$. On the other hand, the image (like $\SUV{V,h}$
  itself, cf.~\cite[6.4.26]{MR1007302}\footnote{ \ %
    The exceptions in~\cite[6.4.26]{MR1007302} are groups over finite
    fields. Apart from symplectic groups and the orthogonal group with
    respect to a quadratic form on $\FF_2^4$ (all of which are of no
    interest in this paper), there is a unitary group on~$\FF_4^2$
    (which is not of interest here, as we assume $\ell=2$ now).  }%
  ) %
  is a perfect group, and the quotient
  $\UV{\Ext[2]V,g}/\EUV{\Ext[2]V,g}$ is abelian (\cite{MR0104764},
  cf.~\cite[6.4.52]{MR1007302}). %
  Thus $\eta(\SUV{V,h})$ coincides with $\EUV{\Ext[2]V,g}$.
\end{proof}

We cannot expect a result similar to~\ref{nonsplitHermitian} in the
characteristic two case, see~\ref{charTwoA3inC3}. %
It remains to study the case where $\sigma=\id$; then
$\alpha=\id$, as well, and we are dealing with a homomorphism
$\eta\colon\SOV{V,h}\to\OV{\Ext[2]V,g}$. We must consider Eichler
transformations now. We may safely ignore the characteristic two case,
cf.~\cite{KramerStroppel-hodge-char2-arXiv}. 

\begin{lemm}\label{SOtrsAbsG}
  If $\sigma=\id$, $\Char\F\ne2$ and~$h$ is isotropic then
  $\eta(\SOV{V,h})$ acts transitively %
  on the set~$\Abs{g}$ of absolute points of the polarity %
  $\pi_g\colon U\leftrightarrow U^{\perp_g}$ on~$\PG{\Ext[2]{V}}$. %
\end{lemm}
\begin{proof}
  Using $\Char{F}\ne2$ and the fact that the isotropic bilinear
  form~$h$ secures the existence of a hyperbolic pair in~$V$, we find
  an orthogonal basis $v_1,v_2,v_3,v_4$ for~$V$ such that $v_1+v_2$ is
  isotropic. Passing to a scalar multiple of~$h$ we may assume
  $h(v_1,v_1)=1$; then $h(v_2,v_2)=-1$. We abbreviate $c\coloneqq
  h(v_3,v_3)$, normalize $b(v_1\wedge v_2\wedge v_3\wedge v_4)=1$ and
  compute %
  $J\left((v_1+v_2)\wedge v_3\right)  = (v_1+v_2)\wedge v_4(-c)$.
  
  From~\ref{lambdaFibres} we know that every absolute point of~$\pi_g$
  is of the form $(u\wedge w)\K$ with $u,w$ in~$V$ such that
  $h(u,u)=0=h(u,w)$.  We may assume $u=v_1+v_2$ because the group
  $\SOV{V,h}$ acts transitively on the set of all one-dimensional
  $h$-isotropic subspaces of~$V$ (by Witt's Theorem).

  So it suffices to remark that
  $(v_1+v_2)\wedge(v_1+v_2)^{\perp_g} = (v_1+v_2)\K$; this follows
  from %
  $\left((v_1+v_2)\wedge v_3\right)(r+\bj_2s) =
  (v_1+v_2)\wedge(v_3r-v_4cs)$ and the observation
  $(v_1+v_2)^{\perp_g} = (v_1+v_2)\F+v_3\F+v_4\F$. %
\end{proof}

\begin{theo}\label{nonsplitBilinear}
  Assume $\sigma=\id$, $\Char\F\ne2$ and that~$\K$ is not split. %
  Then $\eta(\SOV{V,h})$ contains $\EOV{\Ext[2]{V},g}$.
\end{theo}
\begin{proof}
  Recall from~\ref{wittIndexSplit} that~$h$ has Witt index~$1$.

  We show first that Eichler transformations in $\SOV{V,h}$ are mapped
  to Eichler transformations in $\EOV{\Ext[2]{V},g}$.  Consider
  linearly independent $z,w\in V$ with $h(z,z)=0=h(z,w)$. %
  As $X \coloneqq z\F+w\F$ is not totally isotropic, we have
  $h(w,w)\ne0$, and $X^{\perp_h} \cap X = z\F$. %
  Choose $v\in X^{\perp_h}\smallsetminus z\F$ and then
  $u\in\{w,v\}^{\perp_h}$ such that $h(z,u)=1$. Then $z,w,v,u$ are
  linearly independent, in fact, the Gram matrix for~$g$ with respect
  to this basis is $\left(
    \begin{smallmatrix}
      0 & 0 & -2p \\
      0 & -1 & 0 \\
      -2p & 0 & 0
    \end{smallmatrix}\right)$.
  Put $p\coloneqq \frac12\,h(w,w)$. Evaluating at the $\K$-basis
  $z\wedge w$, $z\wedge u$, $w\wedge u$ for $\Ext[2]{V}$ we obtain %
  $\eta(\Sigma_{z,w,p}) = \Sigma_{z\wedge w,z\wedge u,-\frac12}$.

  Now consider an Eichler transformation $\tau\in\EOV{\Ext[2]{V},g}$.
  By~\ref{SOtrsAbsG} we know that $\eta(\SOV{V,h})$ acts transitively
  on the set~$\Abs{g}$ of absolute points of the polarity~$\pi_g$.
  Therefore, we may assume $\tau=\Sigma_{Z,W,q}$ for $Z=z\wedge
  w$. Now $g(Z,W)=0$ implies $W = Za+Uc$ for $U=z\wedge u$ and some
  $a,c\in\K$. We find $\Sigma_{Z,W,q} = \Sigma_{Z,Uc,q}$ for
  $q=\frac12\,g(Uc,Uc) = \frac{c^2}{2}\,g(U,U)$. %
  Note also that $\Sigma_{Z,Uc,q} = \Sigma_{Zc,U,p}$ with $q=c^2\,p$. 
\end{proof}

\begin{rema}\label{rem:spinorNorm}
  If $\sigma=\id$ and the discriminant of~$h$ is not~$1$ then~$\K$ is
  a commutative field, and the group $\EOV{\Ext[2]V,g}$ coincides %
  (see~\cite[9.7]{MR1859189}) with the kernel of the spinor norm
  $\Theta\colon\SOV{\Ext[2]V,g}\to\K^\myTimes/\K^\sqt$ where $\K^\sqt$
  denotes the group of squares in~$\K^\myTimes$. %
  We do not explicitly determine the range of $\eta$ here. %
\end{rema}

\begin{ndef}[Example: Lorentz group]\label{exam:Lorentz}
  Consider $\R=\RR$ and the (essentially, up to a choice of basis,
  unique) symmetric bilinear form~$h$ of Witt index~$1$ on~$\RR^4$; in
  fact, this is (up to multiplying the form by~$-1$, which does not
  affect the corresponding group of isometries) the form used to
  describe special relativity in physics, the group $\OV{\RR^4,h}$ is
  known as the \emph{Lorentz group}. %
  The group $\SOV{\RR^4,h}$ is not connected; its connected component
  is $\EOV{\RR^4,h}$; this subgroup may be interpreted as the subgroup
  of $\SOV{\RR^4,h}$ leaving invariant the direction of the light
  cone. %
  We have a direct product
  $\SOV{\RR^4,h} = \langle-\id\rangle \, \EOV{\RR^4,h}$.
    
  The discriminant of~$h$ is~$-1$, so $\K$ is not split. Then
  $\K\cong\CC$, the unique quadratic extension field of~$\RR$. %
  The form~$g$ is now a non-degenerate symmetric bilinear form
  (see~\ref{def:g}) on $W\cong\CC^3$, and uniquely determined (up to a
  choice of basis); on~$\CC^3$, such a form~$\tilde{g}$ is given as
  the sum of the squares of the
  coordinates. From~\ref{nonsplitBilinear} we know that the image
  of~$\SOV{\RR^4,h}$ under~$\eta$ contains
  $\EOV{W,g} \cong \EOV{\CC^3,\tilde{g}}$, and is contained in
  $\OV{W,g}$. %
  Note that $\eta(\SOV{\RR^4,h}) = \eta(\EOV{\RR^4,h})$; the kernel
  $\langle-\id\rangle$ is a complement to $\EOV{\RR^4,h}$ in
  $\SOV{\RR^4,h}$.

  Moreover, one knows that $\EOV{\CC^3,\tilde{g}}$ coincides with
  $\SOV{\CC^3,\tilde{g}}$; either by a connectivity argument, or by
  the more general argument from~\ref{rem:spinorNorm}.  The
  isomorphism from
  $\EOV{\RR^4,h}/\langle-\id\rangle \cong \EOV{\RR^4,h}$
  onto~$\SOV{\CC^3,\tilde{g}}$ that we obtain here is known as the
  isomorphism between the simple Lie algebras of
  types~$\rType{D}2{\RR,1}$ and $\Type{B}1$, in the notation of
  Tits~\cite{MR0218489}. %
  Note also that $\EOV{\CC^3,\tilde{g}} \cong \PSL[2]\CC$. This is a
  special case of a general result; for real and complex Lie algebras
  it reflects $\rType{B}1{\RR,1}\cong\rType{A}1{\RR}$ and
  $\Type{B}1 \cong \Type{A}1$.
\end{ndef}

\removelastskip%
\enlargethispage{15mm}%
\begin{ndef}[Examples: Orthogonal groups over finite fields of odd
  order]\label{exam:finiteOrth}%
  Let $\F = \FF_e$ be the finite field of order~$e$, and consider a
  diagonalizable $\sigma$-hermitian form~$h$
  on~$\F^4$. If~$\sigma\ne\id$ then~$h$ has discriminant one, and~$\K$
  is split. So assume $\sigma=\id$, and that the discriminant is not a
  square.  Ignoring the inseparable cases, we assume
  $\Char\F\ne2$. Then the assumption that~$\K$ is non-split yields
  that~$h$ is not the hyperbolic form on~$\F^4$, but equivalent to the
  form known as~$q_4^-$ (cp.~\cite[pp.\,139\,ff]{MR1189139}). In other
  words, there exists $\delta\in\F\smallsetminus\F^\sq$ such that we may
  assume that (with respect to the standard basis $e_1,\dots,e_4$) the
  Gram matrix of~$h$ is $T = \left(
    \begin{smallmatrix}
      0 & 1 & 0 & 0 \\
      1 & 0 & 0 & 0 \\
      0 & 0 & 1 & 0 \\
      0 & 0 & 0 &-\delta \\
    \end{smallmatrix}\right)$. %
  We have $\F(\sqrt{\delta}) \cong \FF_{e^2}$.  The norm form of the
  quadratic field extension $\FF_{e^2}|\FF_e$ is surjective: for each
  $s\in\F^\myTimes$, there exist $x,y\in\F$ such that
  $s=x^2-\delta y^2$. In standard coordinates, the matrix $\left(
    \begin{smallmatrix}
      0 & s & 0 & 0 \\
      1 & 0 & 0 & 0 \\
      0 & 0 & x & \delta y \\
      0 & 0 & y & x \\
    \end{smallmatrix}\right)$ %
  now describes a similitude $\gamma_s$ with multiplier
  $r_{\gamma_s}=s$ and $\det\gamma_s=-s^2$.

  With respect to the basis $v_1\coloneqq e_1+e_2$,
  $v_2\coloneqq e_1-e_2$, $v_3\coloneqq e_3$, $v_4\coloneqq e_4$, the
  Gram matrix for~$h$ is $H = \left(
    \begin{smallmatrix}
      2 & 0 & 0 & 0 \\
      0 &-2 & 0 & 0 \\
      0 & 0 & 1 & 0 \\
      0 & 0 & 0 &-\delta
    \end{smallmatrix}\right)$. %
  Using $b$ with $b(v_1\wedge v_2\wedge v_3\wedge v_4) = 2$, we obtain
  $\bj^2=\delta$, and $\K \cong \F(\sqrt\delta)$. %
  In $M_2 \cong \K$ (see~\ref{similitudeAlgebraAut}), conjugation by
  $\gamma_s$ maps~$J$ to
  $\frac{\det\gamma_s}{r_{\gamma_s}^2}\,\id\circ J =
  \frac{-s^2}{s^2}\,\id\circ J = -J$. %
  So conjugation by~$\gamma_s$ induces the generator of the Galois
  group of the extension~$\K|\F$.

  The group $\OV{\F^4,h} \cong \Orth[-]4{e}$ has order
  $2e^2(e^2+1)(e^2-1)$, we have of course
  $|\SOV{\F^4,h}| = e^2(e^2+1)(e^2-1)$. The group $\EOV{\F^4,h}$ has
  index~$2$ in $\SOV{\F^4,h}$, the element $-\id$ generates a
  complement to $\EOV{\F^4,h}$ in $\SOV{\F^4,h}$, see~\cite[9.7,
  9.8]{MR1859189}.  So $\eta(\SOV{\F^4,h}) = \eta(\EOV{\F^4,h})$ is
  isomorphic to~$\EOV{\F^4,h}$, and contains $\EOV{W,g}$. The latter
  group has index~$2$ in $\SOV{W,g} \cong \PGL[2]{\K}$,
  see~\cite[11.8]{MR1189139}. So
  $|\EOV{W,g}| = (e^2+1)e^2(e^2-1)/2 = |\EOV{\F^4,h}|$, and
  $\EOV{\F^4,h} \cong \EOV{W,g} \cong \PSL[2]{\FF_{e^2}}$ follows.

  We note that $\left(
    \begin{smallmatrix}
      1 & 0 & 0 & 0 \\
      0 & 1 & 0 & 0 \\
      0 & 0 & 1 & 0 \\
      0 & 0 & 0 &-1 \\
    \end{smallmatrix}\right)$ describes an element
  $\rho \in \OV{\F^4,h} \smallsetminus \SOV{\F^4,h}$, and induces an
  element $\eta(\rho) \in \OV{W,g} \smallsetminus \SOV{W,g}$.  So we have
  found that $\eta(\OV{\F^4,h})$ contains $\EOV{W,g}$ as a subgroup of
  index~$2$, but does not contain $\SOV{W,g}$.

  We remark that the group $\EOV{V,h}$ is often denoted $\SOpV{V,h}$.
\end{ndef}

\section{Anisotropic forms}
\label{sec:anisotropic}

It remains to study the case where the form~$h$ is anisotropic. %
Recall from~\ref{exthMayBecomeIsotropic} that $\Ext[\ell]{h}$ may
become isotropic.  Also, the form~$g$ may be isotropic,
see~\ref{exam:anisotropicNonSplit} below.

We concentrate on the case $\sigma=\id$ in this section; the hermitian
case is left open in the present paper. %
In the case where $\Char\F\ne2$, the isometry groups of anisotropic
(quadratic) forms are notoriously difficult; in particular, we do not
have the tool of Eichler transformations at our disposal.  It also
happens that $\SOV{\Ext[]V,g}$ has large normal subgroups with large
index (\cite[Ch.\,5, \S\,3]{MR0082463}, \cite{MR0257242},
\cite[6.4.55]{MR1007302}, cp. also~\cite{MR3871471}),
see~\ref{properNormalSubgroup} below. %

If $\Char{\F}=2$ then the isometry group of~$h$ is contained in the
isometry group of the anisotropic quadratic diagonal form (over~$\R$)
mapping $v$ to $h(v,v)$; these groups are trivial. %
The split cases have been treated in~\ref{hQuaternionNorm},
\ref{splitSigmaNotId}, and~\ref{splitSigmaNotIdCharTwo}; %
except for the case where $\Char\F=2$ and $\sigma=\id$. %
Those latter cases are treated in~\cite{KramerStroppel-hodge-char2-arXiv}.
In the present section, we therefore concentrate on the non-split
case, and assume $\Char\F\ne2$.

We show that the homomorphism
$\eta_2\colon \SOV{V,h}=\SOV{\F^4,h} \to \SOV{\Ext[2]{V},g} \cong
\SOV{\K^3,g}$ need not be surjective if~$\K$ is not split and~$h$ is
anisotropic.

Over the field of real numbers, anisotropic bilinear forms are
positively or negatively definite ones, and the discriminant is~$1$ if
the dimension is even. As this leads to split cases, the real case
does not play any role in the present section.

\removelastskip%
\enlargethispage{5mm}%
\begin{exam}\label{exam:anisotropicNonSplit}
  Let $\F=\QQ$ and consider the bilinear form~$h$ on~$\QQ^4$ given by
  \[
  h(x_1,x_2,x_3,x_4) = x_1^2+2\,x_2^2+10\,x_3^2-5\,x_4^2 \,.
  \]
  The discriminant of this form is the square class of~$-1$. Thus the
  algebra $\K_2$ is not split; in fact it is isomorphic to $\QQ(i)$
  where $i$ is a square root of~$-1$.

  Moreover, the form $h$ is anisotropic. In fact, if this were not the
  case we could pick a minimal isotropic vector~$z$ with integer
  entries. %
  Considering the equation $h(z,z)=0$ first modulo~$5$, and using the
  fact that the squares in $\ZZ/5\ZZ$ are $0$, $1$, and $4\equiv-1$,
  we see $z_1,z_2\in 5\ZZ$. Considering the equation modulo~$25$, we
  then see $z_3,z_4\in 25\ZZ$. So~$z$ has to lie in~$5\ZZ^4$,
  contradicting our minimality assumption.

  The form $g$ is isotropic; e.g., for
  $w \coloneqq (v_1\wedge v_3)10-(v_1\wedge v_4)(10+\bj)$ we have
  $g(w,w)=0$, cp.~\ref{gDiagonal}. %
  Thus we have an explicit example of an anisotropic form~$h$ such
  that $\K_2$ is not split and~$g$ is isotropic. %
  
  Starting from the standard basis $e_1,e_2,e_3,e_4$ for $\QQ^4$ we
  get the basis %
  ${e_1\wedge e_2}$, ${e_1\wedge e_3}$, ${e_1\wedge e_4}$ for
  $\Ext[2]{\QQ^4}$ considered as a vector space over $\K_2=\QQ(i)$
  where $\bj=10i$ maps %
  ${e_1\wedge e_2}$, ${e_1\wedge e_3}$, and ${e_1\wedge e_4}$ to %
  $2\,{e_3\wedge e_4}$, $10\,{e_4\wedge e_2}$, %
  and $-5\,{e_2\wedge e_3}$, respectively.

  The group $\OV{V,h}$ does not act transitively on the
  $\SOV{\Ext[2]{V},g}$-orbit
  \[
  g^{\gets}(2) \coloneqq \smallset{X\K}{X\in{\Ext[2]V},\, g(X,X)=2}
  \,. %
  \]
  In order to prove this, we note that %
  the restrictions of~$h$ to the two subspaces
  $L_1\coloneqq(1,0,0,0)^\transp\F\oplus(0,1,0,0)^\transp\F$ and
  $L_2\coloneqq(0,0,1,0)^\transp\F\oplus(1/2,0,0,1/10)^\transp\F$ both
  have discriminant~$2$, and~$\lambda(L_1)$ and~$\lambda(L_2)$ both
  belong to $g^{\gets}(2)$. From~\ref{lambdaFibres} we know %
  that $\lambda(L)=\lambda(L_k)$ means $L\in\{L_k^{},L_k^{\perp_h}\}$
  because the restriction of~$h$ to~$L_k$ is not degenerate. As $L_2$
  is not isometric\footnote{ \ %
    The restrictions of $h$ to $L_1$ and~$L_2$ are similar, but not
    isometric; e.g., $5\notin\set{h(v,v)}{v\in L_2}$.} %
  to any $L\in\{L_1^{},L_1^{\perp_h}\}$, we find that there is no
  element of $\OV{V,h}$ mapping $\lambda(L_1)$ to~$\lambda(L_2)$.
\end{exam}

\begin{lemm}[{\cite[II\,\S6,\,1),\,p.\,49]{MR0072144}}]\label{generateSO}
  If $\dim{X}\ge3$ and $f$ is a non-degenerate symmetric bilinear form
  on~$X$ then $\SOV{X,f}$ is generated by the half-turns, %
  i.e.  the involutions whose space of fixed points has
  codimension~$2$ in~$X$. %
  \qed
\end{lemm}

\begin{theo}\label{properNormalSubgroup}
  The subgroup $H\le \SOV{\Ext[2]{V},g}$ generated by all half-turns
  around axes that meet the Klein quadric is a normal subgroup of\/
  $\gOV{\Ext[2]{V},g}$, and this subgroup~$H$ is contained in the
  image $\eta(\SOV{V,h})$. %
  In general, both~$H$ and\/ $\eta(\SOV{V,h})$ are proper subgroups of
  $\SOV{\Ext[2]{V},g}$. 
\end{theo}
\begin{proof}
  Each half-turn is determined by its axis; thus half-turns are
  conjugates precisely if their axes are in the same orbit under the
  group of similitudes. In order to prove normality of~$H$ in
  $\gOV{\Ext[2]{V},g}$, it remains to note that the Klein quadric is
  invariant under the action of the group of semi-similitudes because
  ${\pf(X,X)=0}$ characterizes the points~$X\F$ on the Klein quadric,
  and $\pf(X,X)=0 \iff g(X,X)\in\F$ holds for $X\in\Ext[2]{V}$ by the
  construction of the form~$g$, see~\ref{def:g}.

  Now let $X\in\Ext[2]{V}$ such that $X\K$ meets the Klein
  quadric. Then there exists $Y\in X\K$ such that $\pf(Y,Y)=0$, and
  there are $v_1,v_2\in V$ such that $Y=v_1\wedge v_2$. 

  Since $h$ is anisotropic we may assume that $h(v_1,v_2)=0$, and we
  can extend this to an orthogonal basis $v_1,v_2,v_3,v_4$
  for~$V$. Then $Y_1\coloneqq v_1\wedge v_2$,
  $Y_2\coloneqq v_1\wedge v_3$, $Y_3\coloneqq v_1\wedge v_4$ form a
  $\K$-basis for $\Ext[2]{V}$. The linear map $\rho_1$ mapping $v_1$
  to $-v_1$ and fixing $v_2,v_3,v_4$ is a hyperplane reflection in
  $\OV{V,h}$. %
  We see immediately that $\eta(\rho_1)=\eta(-\rho_1)$ is a
  $\K$-semilinear map with companion automorphism~$\kappa$, fixing
  $Y_1\bj$, $Y_2\bj$ and $Y_3\bj$. Interchanging the roles of $v_1$
  and $v_2$ we obtain $\rho_2$ such that $\eta(\rho_2\circ\rho_1)$ is
  the $\K$-linear map fixing $Y_1$ and inducing $-\id$ on
  $Y_2\K\oplus Y_3\K$. Thus $\eta(\rho_2\circ\rho_1)$ is one of the
  involutions in $\SOV{\Ext[2]{V},g}$. %
  In this way we obtain all those half-turns in $\SOV{\Ext[2]{V},g}$
  whose axis is spanned by some
  $X\in\smallset{u\wedge v}{u,v\in V}\smallsetminus\{0\}$, i.e., those
  with axes that meet the Klein quadric. %

  Finally, we recall from~\ref{exam:anisotropicNonSplit} that it may
  happen that $\eta(\SOV{V,h})$ does not act transitively on some
  $\SOV{\Ext[2]{V},g}$-orbit. In such cases, we clearly have
  $H\le\eta(\SOV{V,h}) \lneqq \SOV{\Ext[2]{V},g}$.
\end{proof}

In the situation of~\ref{properNormalSubgroup}, it remains as an open
question whether there are examples where $H\lneqq\eta(\SOV{V,h})$.

\enlargethispage{8mm}%
\begin{acks}
  In her diploma thesis~\cite{Fiedler} Ulrike Fiedler worked out
  several examples explicitly, and spotted various slips in a draft of
  the present paper.

  Funded by the Deutsche Forschungsgemeinschaft through a
  Polish-German \emph{Beethoven} grant KR1668/11, and under Germany's
  Excellence Strategy EXC 2044-390685587, Mathematics M\"unster:
  Dynamics-Geometry-Structure.
\end{acks}

\bigbreak

\begin{thebibliography}{10}
\providecommand{\href}[2]{#2}
\providecommand{\eprint}[1]{\href{http://arxiv.org/abs/#1}{#1}}
\providecommand{\bbldiplomarbeit}{Diplomarbeit}
\providecommand{\url}[1]{\href{#1}{#1}}
\providecommand{\urlprefix}{}
\let\oldunderscore_
\catcode`\_=13
\providecommand{\doi}[1]{\href{http://dx.doi.org/#1}{{\def_{\_}\normalfont\ttfamily doi:#1}}\let_\oldunderscore}
\providecommand{\MR}[1]{\relax\ifhmode\unskip\space\fi \MRnumberextract#1 \,}
\def\MRnumberextract#1 #2\,{\MRhref{#1}{#2}}%
\providecommand{\MRhref}[2]{%
  \href{http://www.ams.org/mathscinet-getitem?mr=#1}{MR\,#1 #2}}
\providecommand{\ZBL}[1]{\relax\ifhmode\unskip\space\fi \ZBLhref{#1}}
\providecommand{\ZBLhref}[1]{%
  \href{http://zbmath.org/?q=an:#1}{Zbl #1}}
\providecommand{\JfM}[1]{\relax\ifhmode\unskip\space\fi \JfMhref{#1}}
\providecommand{\JfMhref}[1]{%
  \href{http://zbmath.org/?q=an:#1}{JfM #1}}

\newcommand{\Capitalize}[1]{\uppercase{#1}}
\newcommand{\capitalize}[1]{\expandafter\Capitalize#1}

\bibitem{MR0082463}
E.~Artin, \emph{Geometric algebra}, Interscience Publishers, Inc., New
  York-London, 1957. \MR{0082463.} \ZBL{0077.02101}.

\bibitem{MR0052795}
R.~Baer, \emph{Linear algebra and projective geometry}, Academic Press Inc.,
  New York, 1952. \MR{0052795.} \ZBL{0049.38103}.

\bibitem{MR3871471}
A.~Blunck, N.~Knarr, B.~Stroppel, and M.~J. Stroppel, \emph{Transitive
  groups of similitudes generated by octonions}, J. Group Theory \textbf{21}
  (2018), 1001--1050, \doi{10.1515/jgth-2018-0018}. \MR{3871471.}

  \goodbreak%
\bibitem{MR0026989}
N.~Bourbaki, \emph{\'{E}l\'ements de math\'ematique. {VII}. {P}remi\`ere
  partie: {L}es structures fondamentales de l'analyse. {L}ivre {II}:
  {A}lg\`ebre. {C}hapitre {III}: {A}lg\`ebre multilin\'eaire}, Actualit\'es
  Sci. Ind.  1044, Hermann, Paris, 1948. \MR{0026989 (10,231d).}
  \ZBL{0039.25902}.

\bibitem{MR0107661}
N.~Bourbaki, \emph{\'{E}l\'ements de math\'ematique. {P}remi\`ere partie: {L}es
  structures fondamentales de l'analyse. {L}ivre {II}: {A}lg\`ebre. {C}hapitre
  9: {F}ormes sesquilin\'eaires et formes quadratiques}, Actualit\'es Sci. Ind.
  no. 1272, Hermann, Paris, 1959. \MR{0107661 (21 \#6384).} \ZBL{0102.25503}.

\bibitem{MR0274237}
N.~Bourbaki, \emph{\'{E}l\'ements de math\'ematique. {A}lg\`ebre. {C}hapitres 1
  \`a 3}, Hermann, Paris, 1970. \MR{0274237 (43 \#2).} \ZBL{0211.02401}.

\bibitem{MR979982}
N.~Bourbaki, \emph{Algebra. {I}. {C}hapters 1--3}, Elements of Mathematics
  (Berlin), Springer-Verlag, Berlin, 1989. \MR{979982 (90d:00002).}
  \ZBL{0673.00001}.

\bibitem{MR0072144}
J.~A. Dieudonn{\'e}, \emph{La g\'eom\'etrie des groupes classiques}, Ergebnisse
  der {M}athematik und ihrer {G}renzgebiete ({N}.{F}.) ~5, Springer-Verlag,
  Berlin, 1955. \MR{0072144 (17,236a).} \ZBL{0221.20056}.

\bibitem{MR0310083}
J.~A. Dieudonn{\'e}, \emph{La g\'eom\'etrie des groupes classiques}, Ergebnisse
  der {M}athematik und ihrer {G}renzgebiete ({N}.{F}.) ~5, Springer-Verlag,
  Berlin, 3rd ed., 1971. \MR{0310083 (46 \#9186).}
  \ZBL{0221.20056}.

\bibitem{Fiedler}
U.~Fiedler, \emph{Hodge-{O}peratoren und {A}usnahmeisomorphismen},
  \bbldiplomarbeit{}, Fakult\"at Mathematik, Stuttgart, 2014.

\bibitem{MR1859189}
L.~C. Grove, \emph{Classical groups and geometric algebra}, Graduate Studies in
  Mathematics ~39, American Mathematical Society, Providence, RI, 2002.
  \MR{1859189 (2002m:20071).} \ZBL{0990.20001}.

\bibitem{MR2926161}
M.~Gulde and M.~J. Stroppel, \emph{Stabilizers of subspaces under
  similitudes of the {K}lein quadric, and automorphisms of {H}eisenberg
  algebras}, Linear Algebra Appl. \textbf{437} (2012), 1132--1161,
  \doi{10.1016/j.laa.2012.03.018}, \eprint{arXiv:1012.0502}. \MR{2926161.}
  \ZBL{06053093}.

\bibitem{MR1007302}
A.~J. Hahn and O.~T. O'Meara, \emph{The classical groups and
  {$K$}-theory}, Grundlehren der {M}athematischen {W}issenschaften  291,
  Springer-Verlag, Berlin, 1989. \MR{1007302 (90i:20002).} \ZBL{0683.20033}.

\bibitem{MR514561}
S.~Helgason, \emph{Differential geometry, {L}ie groups, and symmetric spaces},
  Pure and Applied Mathematics ~80, Academic Press Inc., New York, 1978.
  \MR{514561 (80k:53081).} \ZBL{0993.53002}.

\bibitem{MR101269}
L.-K. Hua, \emph{A subgroup of the orthogonal group with respect to an
  indefinite quadratic form}, Sci. Record (N.S.) \textbf{2} (1958), 329--331.
  \MR{101269.} \ZBL{0083.02004}.

\bibitem{MR0001957}
N.~Jacobson, \emph{A note on hermitian forms}, Bull. Amer. Math. Soc.
  \textbf{46} (1940), 264--268, \doi{10.1090/S0002-9904-1940-07187-3}.
  \MR{0001957 (1,325d).} \ZBL{0024.24503}. \JfM{66.0048.03}.

\bibitem{MR780184}
N.~Jacobson, \emph{Basic algebra. {I}}, W. H. Freeman and Company, New York,
  2nd ed., 1985. \MR{780184 (86d:00001).} \ZBL{0557.16001}.

\bibitem{MR0257242}
R.~P. Johnson, \emph{Orthogonal groups of local anisotropic spaces}, Amer. J.
  Math. \textbf{91} (1969), 1077--1105, \doi{10.2307/2373317}. \MR{0257242 (41
  \#1893).} \ZBL{0194.04601}.

\bibitem{MR2257570}
A.~W. Knapp, \emph{Basic algebra}, Cornerstones, Birkh\"auser Boston Inc.,
  Boston, MA, 2006. \MR{2257570 (2007e:00001).} \ZBL{1106.00001}.

\bibitem{MR3761133}
N.~Knarr and M.~J. Stroppel, \emph{Subforms of norm forms of octonion
  fields}, Arch. Math. (Basel) \textbf{110} (2018), 213--224,
  \doi{10.1007/s00013-017-1129-x}. \MR{3761133.} \ZBL{06844626}.

\bibitem{MR1096299}
M.-A. Knus, \emph{Quadratic and {H}ermitian forms over rings}, Grundlehren der
  {M}athematischen {W}issenschaften  294, Springer-Verlag, Berlin, 1991.
  \MR{1096299 (92i:11039).} \ZBL{0756.11008}.

\bibitem{KramerStroppel-hodge-char2-arXiv}
L.~Kramer and M.~J. Stroppel, \emph{Hodge operators and groups of
  isometries of diagonalizable symmetric bilinear forms in characteristic two},
  2022, \doi{10.48550/ARXIV.2208.11326}.
  \urlprefix\url{https://arxiv.org/abs/2208.11326}.

\bibitem{MR2104929}
T.~Y. Lam, \emph{Introduction to quadratic forms over fields}, Graduate Studies
  in Mathematics ~67, American Mathematical Society, Providence, RI, 2005.
  \MR{2104929 (2005h:11075).} \ZBL{1068.11023}.

\bibitem{MR1878556}
S.~Lang, \emph{Algebra}, Graduate Texts in Mathematics  211, Springer-Verlag,
  New York, 3rd ed., 2002. \MR{1878556.} \ZBL{0984.00001}.

\bibitem{MR1308913}
J.~A. Lester, \emph{Orthochronous subgroups of {${\rm O}(p,q)$}}, Linear and
  Multilinear Algebra \textbf{36} (1993), 111--113,
  \doi{10.1080/03081089308818280}. \MR{1308913.} \ZBL{0799.20041}.

\bibitem{MR0401796}
M.~Marcus, \emph{Finite dimensional multilinear algebra. {P}art {II}}, Pure and
  Applied Mathematics ~23, Marcel Dekker Inc., New York, 1975. \MR{0401796.}
  \ZBL{0339.15003}.

\bibitem{MR0506372}
J.~W. Milnor and D.~Husemoller, \emph{Symmetric bilinear forms},
  Ergebnisse der {M}athematik und ihrer {G}renzgebiete ~73, Springer-Verlag,
  New York, 1973. \MR{0506372.} \ZBL{0292.10016}.

\bibitem{MR1763974}
T.~A. Springer and F.~D. Veldkamp, \emph{Octonions, {J}ordan algebras and
  exceptional groups}, Springer Monographs in Mathematics, Springer-Verlag,
  Berlin, 2000. \MR{1763974.} \ZBL{1087.17001}.

\bibitem{MR2431124}
M.~J. Stroppel, \emph{The {K}lein quadric and the classification of nilpotent
  {L}ie algebras of class two}, J. Lie Theory \textbf{18} (2008),
  391--411,
  \urlprefix\url{http://www.heldermann-verlag.de/jlt/jlt18/strola2e.pdf}.
  \MR{2431124.} \ZBL{1179.17013}.

\bibitem{MR2942723}
M.~J. Stroppel, \emph{Orthogonal polar spaces and unitals}, Innov. Incidence
  Geom. \textbf{12} (2011), 167--179,
  \urlprefix\url{http://iig.ugent.be/online/12/volume-12-article-11-online.pdf}.
  \MR{2942723.} \ZBL{06225476}.

\bibitem{MR1189139}
D.~E. Taylor, \emph{The geometry of the classical groups}, Sigma Series in Pure
  Mathematics ~9, Heldermann Verlag, Berlin, 1992. \MR{1189139.}
  \ZBL{0767.20001}.

\bibitem{MR0218489}
J.~Tits, \emph{Tabellen zu den einfachen {L}ie {G}ruppen und ihren
  {D}arstellungen}, Springer-Verlag, Berlin, 1967, \doi{10.1007/BFb0080324}.
  \MR{0218489.} \ZBL{0166.29703}.

\bibitem{MR0470099}
J.~Tits, \emph{Buildings of spherical type and finite {BN}-pairs}, Lecture
  Notes in Mathematics, Vol. 386, Springer, Berlin, 2nd ed.,
  1986. \MR{0470099.} \ZBL{0295.20047}.

\bibitem{MR4279905}
J.~Voight, \emph{Quaternion algebras}, Graduate Texts in Mathematics  288,
  Springer, Cham, 2021, \doi{10.1007/978-3-030-56694-4}. \MR{4279905.}
  \ZBL{1481.11003}.

\bibitem{MR0104764}
G.~E. Wall, \emph{The structure of a unitary factor group}, Inst. Hautes
  \'Etudes Sci. Publ. Math.  (1959), 23 pp. (1959),
  \doi{10.1007/BF02684272}. \MR{0104764.} \ZBL{0087.02202}.

\bibitem{MR722297}
F.~W. Warner, \emph{Foundations of differentiable manifolds and {L}ie groups},
  Graduate Texts in Mathematics ~94, Springer-Verlag, New York-Berlin, 1983.
  \MR{722297.} \ZBL{0516.58001}.

\end{thebibliography}
\providecommand{\noopsort}[1]{}\def\cprime{$'$}
  \def\polhk#1{\setbox0=\hbox{#1}{\ooalign{\hidewidth
  \lower1.5ex\hbox{`}\hidewidth\crcr\unhbox0}}}

\bigskip

\begin{minipage}{0.3\linewidth}
  Linus Kramer 
\end{minipage}
\begin{minipage}[t]{0.6\linewidth}
  Mathematisches Institut\\
  \mbox{Fachbereich Mathematik und Informatik}\\
  Universit\"at M\"unster\\
  Einsteinstra{\ss}e 62\\
  48149 M\"unster (Germany)\\ 
  linus.kramer@uni-muenster.de 
\end{minipage}

\bigskip
\begin{minipage}{0.3\linewidth}
  Markus J. Stroppel 
\end{minipage}
\begin{minipage}[t]{0.6\linewidth}
  LExMath\\
  Fakult\"at 8\\
  Universit\"at Stuttgart\\
  70550 Stuttgart\\ 
  stroppel@mathematik.uni-stuttgart.de 
\end{minipage}

\end{document}